\documentclass[12pt]{amsart}

\usepackage{amsmath}
\usepackage{amssymb}
\usepackage{ifthen}
\usepackage[abbrev]{amsrefs}
\usepackage{framed} \usepackage{float} \usepackage[all,2cell]{xy} \UseAllTwocells
\usepackage{tikz}
\usetikzlibrary{decorations.markings}

\usepackage{geometry}
\tikzset{
 ->-/.style = {
  decoration = {markings, mark=at position #1 with {\arrow[scale=1.5]{>}}},
  postaction = {decorate}
 }
}

\numberwithin{subsection}{section}

\numberwithin{equation}{section}

\newtheorem{thm}{Theorem}[section]
\newtheorem{prop}[thm]{Proposition}
\newtheorem{cor}[thm]{Corollary}
\newtheorem{lem}[thm]{Lemma}

\theoremstyle{remark}
\newtheorem{rem}[thm]{Remark}
\newtheorem{ex}[thm]{Example}

\theoremstyle{definition}
\newtheorem{defn}[thm]{Definition}
\newtheorem{nota}[thm]{Notation}

\newcommand{\F}{\mathbb{F}}

\newcommand{\Z}{\mathbb{Z}}

\newcommand{\cat}[1]{\mathbf{\mathcal{#1}}} \newcommand{\CC}{\mathbf{\mathcal{C}}}

\newcommand{\KK}{\mathbf{\mathcal{K}}}
\newcommand{\OO}{\mathbf{\mathcal{O}}}

\newcommand{\steen}{\mathfrak{A}}
\newcommand{\TT}{\mathbf{\mathcal{T}}}

\newcommand{\al}{\alpha}
\newcommand{\be}{\beta}
\newcommand{\De}{\Delta}
\newcommand{\de}{\delta}
\newcommand{\del}{\partial}
\newcommand{\ep}{\epsilon}
\newcommand{\Ga}{\Gamma}
\newcommand{\ga}{\gamma}
\newcommand{\la}{\lambda}

\newcommand{\phy}{\varphi}
\newcommand{\Si}{\Sigma}

\newcommand{\bu}{\bullet}
\newcommand{\inj}{\hookrightarrow}

\newcommand{\pb}{\ar@{}[dr]|{\mbox{\huge $\lrcorner$}}}
\newcommand{\po}{\ar@{}[dr]|{\mbox{\huge $\ulcorner$}}}
\newcommand{\ral}{\xrightarrow} \newcommand{\Ra}{\Rightarrow}
\newcommand{\surj}{\twoheadrightarrow}
\newcommand{\tild}{\widetilde}

\newcommand{\dfn}{:=}
\newcommand{\inv}{\boxminus}
\newcommand{\lan}{\left\langle}

\newcommand{\ot}{\otimes}
  \newcommand{\ran}{\right\rangle}
\newcommand{\sm}{\wedge}
\newcommand{\sq}{\square \,}

\newcommand{\x}{\times}

\newcommand{\ttA}{\mathbf{Alg_{(1,2)}}}
\newcommand{\Ab}{\mathbf{Ab}}

\newcommand{\ct}[1]{\mathbf{#1}} \newcommand{\dgA}[1][]{\ifthenelse{\equal{#1}{}}{\mathbf{dgA}}{\mathbf{dgA} \lan #1 \ran}}
\newcommand{\DGpd}{\mathbf{Gpd_{(1,2)}}}
\newcommand{\finMod}{\mathbf{Mod}^{\textrm{fin}}}
\newcommand{\gA}[1][]{\ifthenelse{\equal{#1}{}}{\mathbf{gA}}{\mathbf{gA} \lan #1 \ran}}
\newcommand{\Gpd}{\mathbf{Gpd}}
\newcommand{\Ho}{\mathrm{Ho}}

\newcommand{\Mod}{\mathbf{Mod}}

\newcommand{\Spec}{\mathbf{Spec}}
\newcommand{\Topp}{\mathbf{Top}}

\DeclareMathOperator{\AlgCub}{\Theta}
\DeclareMathOperator{\Aut}{Aut}
\DeclareMathOperator{\Comp}{Comp}
\DeclareMathOperator{\Ext}{Ext}
\DeclareMathOperator{\Hom}{Hom}
\DeclareMathOperator{\im}{im}

\DeclareMathOperator{\Nul}{Nul}

\DeclareMathOperator{\Ob}{Ob}
\DeclareMathOperator{\Star}{Star}

\newcommand{\bigpd}{\mathrm{BiGpd}}

\newcommand{\id}{\mathrm{id}}

\newcommand{\opp}{\mathrm{op}}

\newcommand{\sk}{\mathrm{sk}}

\newcommand{\Def}{\textbf}

\newdir{->}{!/5pt/\dir{-}!/2.5pt/\dir{-}*!/-5pt/\dir2{>}}

\begin{document}

\title{$2$-track algebras and the Adams spectral sequence} 

\author{Hans-Joachim Baues}
\email{baues@mpim-bonn.mpg.de}
\address{Max-Planck-Institut f\"ur Mathematik\\
Vivatsgasse 7\\
53111 Bonn\\
Germany}

\author{Martin Frankland}
\email{franklan@mpim-bonn.mpg.de}
\address{Max-Planck-Institut f\"ur Mathematik\\
Vivatsgasse 7\\
53111 Bonn\\
Germany}



\date{\today}
%\date{November 12, 2015}

\begin{abstract}
In previous work of the first author and Jibladze, the $E_3$-term of the Adams spectral sequence was described as a secondary derived functor, defined via secondary chain complexes in a groupoid-enriched category. This led to computations of the $E_3$-term using the algebra of secondary cohomology operations. In work with Blanc, an analogous description was provided for all higher terms $E_r$. In this paper, we introduce $2$-track algebras and tertiary chain complexes, and we show that the $E_4$-term of the Adams spectral sequence is a tertiary Ext group in this sense. This extends the work with Jibladze, while specializing the work with Blanc in a way that should be more amenable to computations.
\end{abstract}

\dedicatory{Dedicated to Ronnie Brown on the occasion of his eightieth birthday.}

\maketitle

\tableofcontents

\section{Introduction}

A major problem in algebraic topology consists of computing homotopy classes of maps between spaces or spectra, notably the stable homotopy groups of spheres $\pi_*^S(S^0)$. One of the most useful tools for such computations is the Adams spectral sequence \cite{Adams58} (and its unstable analogues \cite{BousfieldK72}), based on ordinary mod $p$ cohomology. Given finite spectra $X$ and $Y$, Adams constructed a spectral sequence of the form:
\[
E_2^{s,t} = \Ext_{\steen}^{s,t} \left( H^*(Y; \F_p), H^*(X;\F_p) \right) \Ra [\Si^{t-s} X,Y^{\wedge}_p]
\]
where $\steen$ is the mod $p$ Steenrod algebra, consisting of primary stable mod $p$ cohomology operations, and $Y^{\wedge}_p$ denotes the $p$-completion of $Y$. In particular, taking sphere spectra $X = Y = S^0$, one obtains a spectral sequence
\[
E_2^{s,t} = \Ext_{\steen}^{s,t} \left(\F_p, \F_p \right) \Ra \pi_{t-s}^S(S^0)^{\wedge}_p
\]
abutting to the $p$-completion of the stable homotopy groups of spheres. 
The differential $d_r$ is determined by $r^{\text{th}}$ order cohomology operations \cite{Maunder64}. 
In particular, secondary cohomology operations determine the differential $d_2$ and thus the $E_3$-term. The algebra of secondary operations was studied in \cite{Baues06}. In \cite{BauesJ06}, the first author and Jibladze developed secondary chain complexes and secondary derived functors, and showed that the Adams $E_3$-term is given by secondary $\Ext$ groups of the secondary cohomology of $X$ and $Y$. They used this in \cite{BauesJ11}, along with the algebra of secondary operations, to construct an algorithm that computes the differential $d_2$.

Primary operations in mod $p$ cohomology are encoded by the homotopy category $\Ho (\KK)$ of the Eilenberg-MacLane mapping theory $\KK$, consisting of finite products of Eilenberg-MacLane spectra of the form $\Si^{n_1} H\F_p \x \cdots \x \Si^{n_k} H\F_p$. 
More generally, the $n^{\text{th}}$ Postnikov truncation $P_n \KK$ of the Eilenberg-MacLane mapping theory encodes operations of order up to $n+1$, which in turn determine the Adams differential $d_{n+1}$ and thus the $E_{n+2}$-term \cite{BauesB10}. However, $P_n \KK$ contains too much information for practical purposes. In \cite{Baues15}, the first author and Blanc extracted from $P_n \KK$ the information needed in order to compute the Adams differential $d_{n+1}$. The resulting algebraic-combinatorial structure is called an \emph{algebra of left $n$-cubical balls}.

In this paper, we specialize the work of \cite{Baues15} to the case $n=2$. Our goal is to provide an alternate structure which encodes an algebra of left $2$-cubical balls, but which is more algebraic in nature and better suited for computations. The combinatorial difficulties in an algebra of left $n$-cubical balls arise from triangulations of the sphere $S^{n-1} = \del D^n$. In the special case $n=2$, triangulations of the circle $S^1$ are easily described, unlike in the case $n > 2$. Our approach also extends the work in \cite{BauesJ06} from secondary chain complexes to tertiary chain complexes.

\subsection*{Organization and main results}

We define the notion of $2$-track algebra (Definition~\ref{def:2TrackAlg}) and show that each $2$-track algebra naturally determines an algebra of left $2$-cubical balls (Theorem~\ref{2TrackAlgBalls}). Building on \cite{Baues15}, we show that higher order resolutions always exist in a $2$-track algebra (Theorem~\ref{thm:Resolution}). We show that a suitable $2$-track algebra related to the Eilenberg-MacLane mapping theory recovers the Adams spectral sequence up to the $E_4$-term (Theorem~\ref{thm:d3ASS}). We show that the spectral sequence only depends on the weak equivalence class of the $2$-track algebra (Theorem~\ref{thm:WeakEqClass}).

\begin{rem}
This last point is important in view of the strictification result for secondary cohomology operations: these can be encoded by a graded pair algebra $B_*$ over $\Z/p^2$ \cite{Baues06}*{\S 5.5}. The secondary $\Ext$ groups of the $E_3$-term turn out to be the usual $\Ext$ groups over $B_*$ \cite{BauesJ11}*{Theorem 3.1.1}, a key fact for computations. We conjecture that a similar strictification result holds for tertiary operations, i.e., in the case $n=2$.
\end{rem}

Appendix~\ref{sec:Models2Types} explains why $2$-track groupoids are not models for homotopy $2$-types, and how to extract the underlying $2$-track groupoid from a bigroupoid or a double groupoid.

\subsection*{Acknowledgments}

We thank the referee for their helpful comments. The second author thanks the Max-Planck-Institut f\"ur Mathematik Bonn for its generous hospitality, as well as David Blanc, Robert Bruner, Dan Christensen, and Dan Isaksen for useful conversations.

\section{Cubes and tracks in a space}

In this section, we fix some notation and terminology regarding cubes and groupoids.

\begin{defn} \label{def:Cubes}
Let $X$ be a topological space.

An \Def{$n$-cube} in $X$ is a map $a \colon I^n \to X$, where $I = [0,1]$ is the unit interval. For example, a $0$-cube in $X$ is a point of $X$, and a $1$-cube in $X$ is a path in $X$.

An $n$-cube can be restricted to $(n-1)$-cubes along the $2n$ faces of $I^n$. For $1 \leq i \leq n$, denote:
\begin{align*}
&d_i^0(a) = a \text{ restricted to } I \x I \x \ldots \x \overbrace{\{0\}}^{i} \x \ldots \x I \\
&d_i^1(a) = a \text{ restricted to } I \x I \x \ldots \x \overbrace{\{1\}}^{i} \x \ldots \x I.
\end{align*}

An \Def{$n$-track} in $X$ is a homotopy class, relative to the boundary $\del I^n$, of an $n$-cube. If $a \colon I^n \to X$ is an $n$-cube in $X$, denote by $\{ a \}$ the corresponding $n$-track in $X$, namely the homotopy class of $a$ rel $\del I^n$.
\end{defn}

In particular, for $n=1$, a $1$-track $\{a\}$ is a path homotopy class, i.e., a morphism in the fundamental groupoid of $X$ from $a(0)$ to $a(1)$. Let us fix our notation regarding groupoids. In this paper, we consider only \emph{small} groupoids.

\begin{nota}
A \Def{groupoid} is a (small) category in which every morphism is invertible. Denote the data of a groupoid by $G = \left( G_0, G_1, \de_0, \de_1, \id^{\sq}, \sq, (-)^{\inv} \right)$, where:
\begin{itemize}
 \item $G_0 = \Ob(G)$ is the set of objects of $G$.
 \item $G_1 = \Hom(G)$ is the set of morphisms of $G$. The set of morphisms from $x$ to $y$ is denoted $G(x,y)$. We write $x \in G$ and $\deg(x) = 0$ for $x \in G_0$, and $\deg(x) = 1$ for $x \in G_1$.
 \item $\de_0 \colon G_1 \to G_0$ is the source map.
 \item $\de_1 \colon G_1 \to G_0$ is the target map.
 \item $\id^{\sq} \colon G_0 \to G_1$ sends each object $x$ to its corresponding identity morphism $\id^{\sq}_{x}$.
 \item $\sq \colon G_1 \x_{G_0} G_1 \to G_1$ is composition in $G$.
 \item $f^{\inv} \colon y \to x$ is the inverse of the morphism $f \colon x \to y$.
\end{itemize}
Groupoids form a category $\Gpd$, where morphisms are functors between groupoids.

For any object $x \in G_0$, denote by $\Aut_G(x) = G(x,x)$ the automorphism group of $x$.

Denote by $\Comp(G) = \pi_0(G)$ the components of $G$, i.e., the set of isomorphism classes of objects $G_0 / \sim$.

Denote the fundamental groupoid of a topological space $X$ by $\Pi_{(1)}(X)$.
\end{nota}

\begin{defn}
Let $X$ be a pointed space, with basepoint $0 \in X$. The constant map $0 \colon I^n \to X$ with value $0 \in X$ is called the \Def{trivial $n$-cube}.

A \Def{left $1$-cube} or \Def{left path} in $X$ is a map $a \colon I \to X$ satisfying $a(1) = 0$, that is, $d_1^1(a) = 0$, the trivial $0$-cube. In other words, $a$ is a path in $X$ from a point $a(0)$ to the basepoint $0$. We denote $\de a = a(0)$.

A \Def{left $2$-cube} in $X$ is a map $\al \colon I^2 \to X$ satisfying $\al(1,t) = \al(t,1) = 0$ for all $t \in I$, that is, $d_1^1(\al) = d_2^1(\al) = 0$, the trivial $1$-cube.

More generally, a \Def{left $n$-cube} in $X$ is a map $\al \colon I^n \to X$ satisfying $\al(t_1, \ldots, t_n) = 0$ whenever some coordinate satisfies $t_i = 1$. In other words, for all $1 \leq i \leq n$ we have $d_i^1(\al) = 0$, the trivial $(n-1)$-cube.

A \Def{left $n$-track} in $X$ is a homotopy class, relative to the boundary $\del I^n$, of a left $n$-cube.
\end{defn}

The equality $I^{m+n} = I^m \x I^n$ allows us to define an operation on cubes.

\begin{defn} \label{def:ProdCubes}
Let $\mu \colon X \x X' \to X''$ be a map, for example a composition map in a topologically enriched category $\CC$. For $m,n \geq 0$, consider cubes
\begin{align*}
&a \colon I^m \to X \\
&b \colon I^n \to X'.
\end{align*}
The \Def{$\ot$-composition} of $a$ and $b$ is the $(m+n)$-cube $a \ot b$ defined as the composite
\begin{equation} \label{TensorCubes}
a \ot b \colon I^{m+n} = I^m \x I^n \ral{a \x b} X \x X' \ral{\mu} X''. 
\end{equation}

For $m=n$, the \Def{pointwise composition} of $a$ and $b$ is the $n$-cube defined as the composite
\begin{equation} \label{PointwiseCubes}
ab \colon I^{n} \ral{(a,b)} X \x X' \ral{\mu} X''. 
\end{equation}
The pointwise composition is the restriction of the $\ot$-composition along the diagonal:
\[
\xymatrix{
I^n \ar[r]^-{\De} \ar@/_2pc/[rr]_{ab} & I^n \x I^n \ar[r]^-{a \ot b} & X''.
}
\]
\end{defn}

\begin{rem} \label{RelationsCompo}
For $m=n=0$, the $0$-cube $x \ot y = xy$ is the pointwise composition, which is the composition in the underlying category. For higher dimensions, there are still relations between the $\ot$-composition and the pointwise composition. In suggestive formulas, pointwise composition of paths is given by $(ab)(t) = a(t) b(t)$ for all $t \in I$, whereas the $\ot$-composition of paths is the $2$-cube given by $(a \ot b)(s,t) = a(s) b(t)$.
\end{rem}

Assume moreover that $\mu$ satisfies
\[
\mu(x,0) = \mu(0,x') = 0
\]
for the basepoints $0 \in X, 0 \in X', 0 \in X''$. For example, $\mu$ could be the composition map in a category $\CC$ enriched in $(\Topp_*, \sm)$, the category of pointed topological spaces with the smash product as monoidal structure. If $a$ and $b$ are left cubes, then $a \ot b$ and $ab$ are also left cubes.

\section{$2$-track groupoids}

We now focus on left $2$-tracks in a pointed space $X$, and observe that they form a groupoid. Define the groupoid $\Pi_{(2)}(X)$ with object set:
\[
\Pi_{(2)}(X)_0 = \text{ set of left $1$-cubes in $X$}
\]
and morphism set:
\[
\Pi_{(2)}(X)_1 = \text{ set of left $2$-tracks in $X$}
\]
where the source $\de_0$ and target $\de_1$ of a left $2$-track $\al \colon I \x I \to X$ are given by restrictions
\begin{align*}
\de_0(\al) = d_1^0(\al) \\
\de_1(\al) = d_2^0(\al)
\end{align*}
and note in particular $\de \de_0(\al) = \de \de_1(\al) = \al(0,0)$. In other words, a morphism $\al$ from $a$ to $b$ looks like this:
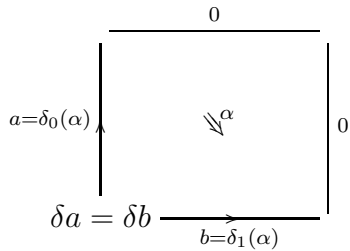
\[
\xymatrix @C=5pc @R=5pc {
\ar@{-}[r]^0 & \ar@{-}[d]^0 \\
\de a = \de b \ar@{-}[u]|-*=0@{>}^-{a=\de_0(\al)} \ar@{-}[r]|-*=0@{>}_-{b = \de_1(\al)} \urtwocell<\omit>{\al} & \\
}
\]

\begin{rem} \label{rem:Globular}
Up to reparametrization, a left $2$-track $\al \colon a \Ra b$ corresponds to a path homotopy from $a$ to $b$, which can be visualized in a globular picture:
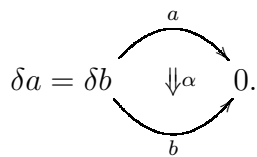
\[
\xymatrix{
**[l] \de a = \de b \rtwocell<8>^a_b{\al} & 0. \\
}
\]
However, the $\ot$-composition will play an important role in this paper, which is why we adopt a cubical approach, rather than globular or simplicial.
\end{rem}

Composition $\be \sq \al$ of left $2$-tracks is described by the following picture:
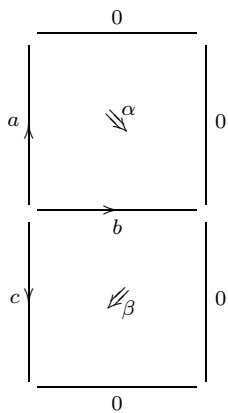
\begin{equation} \label{eq:Gluing2Tracks}
\xymatrix @C=5pc @R=5pc {
\ar@{-}[r]^0 & \ar@{-}[d]^0 \\
\ar@{-}[d]|-*=0@{>}_c \ar@{-}[u]|-*=0@{>}^a \ar@{-}[r]|-*=0@{>}_b \urtwocell<\omit>{\al} \drtwocell<\omit>{\be} & \ar@{-}[d]^0 \\
\ar@{-}[r]_0 & \\
}
\end{equation}

\begin{rem}
To make this definition precise, let $\al \colon a \Ra b$ and $\be \colon b \Ra c$ be left $2$-tracks in $X$, i.e., composable morphisms in $\Pi_{(2)}(X)$. Choose representative maps $\tild{\al}, \tild{\be} \colon I^2 \to X$. Consider the map $f_{\al,\be} \colon [0,1] \x [-1,1] \to X$ pictured in \eqref{eq:Gluing2Tracks}. That is, define
\[
f(s,t) = \begin{cases}
\tild{\al}(s,t) &\text{if } 0 \leq t \leq 1 \\
\tild{\be}(-t,s) &\text{if } -1 \leq t \leq 0. \\
\end{cases}
\]
Now consider the reparametrization map $w \colon I^2 \to [0,1] \x [-1,1]$ illustrated in this picture:
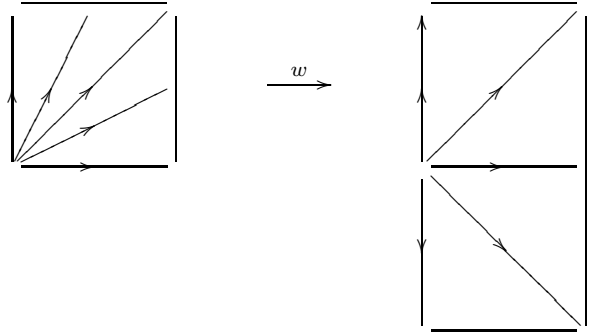
\[
\xymatrix{
\ar@{-}[rr] & & & & & \ar@{-}[rr] & & \\
& & & \ar[r]^w & & & & \\
\ar@{-}[uu]|-*=0@{>} \ar@{-}[uur]|-*=0@{>} \ar@{-}[uurr]|-*=0@{>} \ar@{-}[urr]|-*=0@{>} \ar@{-}[rr]|-*=0@{>} & & \ar@{-}[uu] & & & \ar[uu]|-*=0@{>} \ar@{-}[uurr]|-*=0@{>} \ar@{-}[rr]|-*=0@{>} \ar@{-}[ddrr]|-*=0@{>} \ar@{-}[dd]|-*=0@{>} & & \\
& & & & & & & \\
& & & & & \ar@{-}[rr] & & \ar@{-}[uuuu] \\
}
\]
Explicitly, the restriction $w \vert_{\del I^2}$ to the boundary is the piecewise linear map satisfying
\[
\begin{cases}
w(0,0) = (0,0) \\
w(0,1) = (0,1) \\
w(\frac{1}{2},1) = (1,1) \\
w(1,1) = (1,0) \\
w(1, \frac{1}{2}) = (1,-1) \\
w(1,0) = (0,-1) \\
\end{cases}
\]
and $w(x)$ is defined for points $x \in I^2$ in the interior as follows. Write $x = p (0,0) + q y$ as a unique convex combination of $(0,0)$ and a point $y$ on the boundary $\del I^2$. Then define $w(x) = p w(0,0) + q w(y) = q w(y)$. Finally, the composition $\be \sq \al \colon a \Ra c$ is $\{ f_{\al,\be} \circ w \}$, the homotopy class of the composite
\[
\xymatrix{
I^2 \ar[r]^-w & [0,1] \x [-1,1] \ar[r]^-{f_{\al,\be}} & X \\
}
\]
relative to the boundary $\del I^2$.

In other notation, we have inclusions $d_2^0 \colon I^1 \inj I^2$ as the bottom edge $I \x \{0\}$ and $d_1^0 \colon I^1 \inj I^2$ as the left edge $\{0\} \x I$, our $w$ is a map $w \colon I^2 \to I^2 \cup_{I^1} I^2$, and $\be \sq \al$ is the homotopy class of the composite
\[
\xymatrix{
I^2 \ar[r]^-w & I^2 \cup_{I^1} I^2 \ar[r]^-{\left[ \al \: \be \right]} & X. \\
}
\]
\end{rem}

Given a left path $a$ in $X$, the identity of $a$ in the groupoid $\Pi_{(2)}(X)$ is the left $2$-track is pictured here:
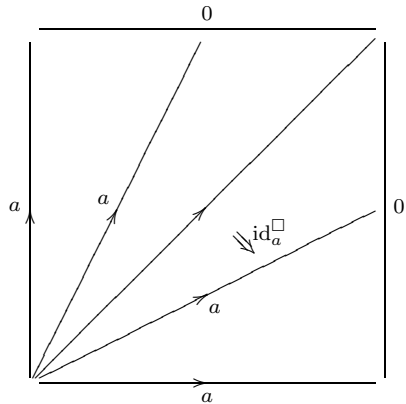
\[
\xymatrix @C=5pc @R=5pc {
\ar@{-}[rr]^0 & & \ar@{-}[dd]^0 \\
& & \\
\ar@{-}[uu]|-*=0@{>}^a \ar@{-}[uur]|-*=0@{>}^a \ar@{-}[uurr]|-*=0@{>}^{} \ar@{-}[urr]|-*=0@{>}_{a} \ar@{-}[rr]|-*=0@{>}_a \uurrtwocell<\omit>{<4>\hspace{1em}\id^{\sq}_a} & & \\
}
\]
More precisely, for points $x \in I^2$ in the interior, write $x = p (0,0) + q y$ as a unique convex combination of $(0,0)$ and a point $y$ on the boundary $\del I^2$. Then define $\id^{\sq}_a(x) = a(q)$.

The inverse $\al^{\inv} \colon b \Ra a$ of a left $2$-track $\al \colon a \Ra b$ is the homotopy class of the composite $\al \circ T$, where $T \colon I^2 \to I^2$ is the map swapping the two coordinates: $T(x,y) = (y,x)$.

\begin{lem}
Given a pointed topological space $X$, the structure described above makes $\Pi_{(2)}(X)$ into a groupoid, called the \Def{groupoid of left $2$-tracks} in $X$.
\end{lem}

\begin{proof}
Standard.
\end{proof}

\begin{defn} \label{def:StrictlyAbel}
A groupoid $G$ is \Def{abelian} if the group $\Aut_G(x)$ is abelian for every object $x \in G_0$. The groupoid $G$ is \Def{strictly abelian} if it is pointed (with basepoint $0 \in G_0$), and is equipped with a family of isomorphisms
\[
\psi_x \colon \Aut_G(x) \ral{\simeq} \Aut_G(0)
\]
indexed by all objects $x \in G_0$, such that the diagram 
\begin{equation} \label{eq:CompatIso}
\xymatrix{
\Aut_G(y) \ar[dr]_{\psi_y} \ar[r]^{\phy^f} & \Aut_G(x) \ar[d]^{\psi_x} \\
& \Aut_G(0)
}
\end{equation}
commutes for every map $f \colon x \to y$ in $G$, where $\phy^f$ denotes the ``change of basepoint'' isomorphism
\begin{align*}
\phy^f \colon \Aut_G(y) &\ral{\simeq} \Aut_G(x) \\
\al &\mapsto \phy^f(\al) = f^{\inv} \sq \al \sq f.
\end{align*}
\end{defn}

\begin{rem}
A strictly abelian groupoid is automatically abelian. Indeed, the compatibility condition \eqref{eq:CompatIso} applied to automorphisms $f \colon 0 \to 0$ implies that conjugation $\phy^f \colon \Aut_G(0) \to \Aut_G(0)$ is the identity.
\end{rem}

\begin{defn}
A groupoid $G$ is \Def{pointed} if it has a chosen basepoint, i.e., an object $0 \in G_0$. Here $0$ is an abuse of notation: the basepoint is not assumed to be an initial object for $G$.

The \Def{star} of a pointed groupoid $G$ is the set of all morphisms to the basepoint $0$, denoted by:
\[
\Star(G) = \left\{ f \in G_1 \mid \de_1(f) = 0 \right\}.
\]
For a morphism $f \colon x \to 0$ in $\Star(G)$, we write $\de f = \de_0 f = x$.
 
If $G$ has a basepoint $0 \in G_0$, then we take $\id^{\sq}_{0} \in G_1$ as basepoint for the set of morphisms $G_1$ and for $\Star(G) \subseteq G_1$; we sometimes write $0 = \id^{\sq}_{0}$. Moreover, we take the component of the basepoint $0$ as basepoint for $\Comp(G)$, the set of components of $G$.
\end{defn}

\begin{prop}
$\Pi_{(2)}(X)$ is a strictly abelian groupoid, and it satisfies $\Comp \Pi_{(2)}(X) \simeq \Star \Pi_{(1)}(X)$.
\end{prop}

\begin{proof}
Let $a \in \Pi_{(2)}(X)_0$ be a left path in $X$. To any automorphism $\al \colon 0 \Ra 0$ in $\Pi_{(2)}(X)$, one can associate the well-defined left $2$-track indicated by the picture
\begin{equation} \label{StrucIso}
\xymatrix @C=5pc @R=5pc{
& & 0 \\
0 \ar@{-}[urr]^0 \urrtwocell<\omit>{<5>\al} \ar@{-}[r]^0 & 0 \ar@{-}[d]^0 & \\
\de a \ar@{-}[u]|-*=0@{>}^a \ar@{-}[r]|-*=0@{>}_a \urtwocell<\omit>{\quad \id^{\sq}_a} & 0 \ar@{-}[uur]_0 & \\
}
\end{equation}
which is a morphism $a \Ra a$. This assignment defines a map $\Aut_{\Pi_{(2)}(X)}(0) \to \Aut_{\Pi_{(2)}(X)}(a)$ and is readily seen to be a group isomorphism, whose inverse we denote $\psi_a$. One readily checks that the family $\psi_a$ is compatible with change-of-basepoint isomorphisms.

The set $\Comp \Pi_{(2)}(X)$ is the set of left paths in $X$ quotiented by the relation of being connected by a left $2$-track. The set $\Star \Pi_{(1)}(X)$ is the set of left paths in $X$ quotiented by the relation of path homotopy. But two left paths are path-homotopic if and only if they are connected by a left $2$-track.
\end{proof}

The bijection $\Comp \Pi_{(2)}(X) \simeq \Star \Pi_{(1)}(X)$ is induced by taking the homotopy class of left $1$-cubes. Consider the function $q \colon \Pi_{(2)}(X)_0 \to \Pi_{(1)}(X)_1$ which sends a left $1$-cube to its left $1$-track $q(a) = \{ a \}$. Then the image of $q$ is $\Star \Pi_{(1)}(X) \subseteq \Pi_{(1)}(X)_1$ and $q$ is constant on the components of $\Pi_{(2)}(X)_0$. We now introduce a definition based on those features of $\Pi_{(2)}(X)$.

\begin{defn} \label{def:2TrackGpd}
A \Def{$2$-track groupoid} $G = (G_{(1)}, G_{(2)})$ consists of:
\begin{itemize}
\item Pointed groupoids $G_{(1)}$ and $G_{(2)}$, with $G_{(2)}$ strictly abelian.
\item A pointed function $q \colon G_{(2)0} \surj \Star G_{(1)}$ which is constant on the components of $G_{(2)}$, and such that the induced function $q \colon \Comp G_{(2)} \ral{\simeq} \Star G_{(1)}$ is bijective.
\end{itemize}
We assign degrees to the following elements:
\[
\deg(x) = \begin{cases}
0 \text{ if } x \in G_{(1)0} \\
1 \text{ if } x \in G_{(2)0} \\
2 \text{ if } x \in G_{(2)1} \\
\end{cases}
\]
and we write $x \in G$ in each case.

A \Def{morphism of $2$-track groupoids} $F \colon G \to G'$ consists of a pair of pointed functors
\begin{align*}
&F_{(1)} \colon G_{(1)} \to G'_{(1)} \\
&F_{(2)} \colon G_{(2)} \to G'_{(2)} 
\end{align*}
satisfying the following two conditions.
\begin{enumerate}
\item (\textit{Structural isomorphisms}) For every object $a \in G_{(2)0}$, the diagram
\[
\xymatrix{
\Aut_{G_{(2)}}(a) \ar[d]_{\psi_a} \ar[r]^-{F_{(2)}} & \Aut_{G'_{(2)}}( F_{(2)} a) \ar[d]^{\psi_{F_{(2)} a}} \\
\Aut_{G_{(2)}}(0) \ar[r]_-{F_{(2)}} & \Aut_{G'_{(2)}}(0') \\
}
\]
commutes.
\item (\textit{Quotient functions}) The diagram
\[
\xymatrix{
G_{(2)0} \ar@{->>}[d]_q \ar[r]^-{F_{(2)}} & G'_{(2)0} \ar@{->>}[d]^{q'} \\
\Star G_{(1)} \ar[r]_-{F_{(1)}} & \Star G'_{(1)} \\
}
\]
commutes.
\end{enumerate}
Let $\DGpd$ denote the category of $2$-track groupoids.
\end{defn}

\begin{rem}
If $\al \colon a \Ra b$ is a left $2$-track in a space, then the left paths $a$ and $b$ have the same starting point $\de a = \de b$. This condition is encoded in the definition of $2$-track groupoid. Indeed, if $\al \colon a \Ra b$ is a morphism in $G_{(2)}$, then $a,b \in G_{(2)0}$ belong to the same component of $G_{(2)}$. Thus, we have $q(a) = q(b) \in \Star G_{(1)}$ and in particular $\de q(a) = \de q(b) \in G_{(1)0}$.
\end{rem}

\begin{defn}
The \Def{fundamental $2$-track groupoid} of a pointed space $X$ is
\[
\Pi_{(1,2)}(X) \dfn \left( \Pi_{(1)}(X), \Pi_{(2)}(X)\right).
\]
This construction defines a functor $\Pi_{(1,2)} \colon \Topp_* \to \DGpd$.
\end{defn}

\begin{rem} \label{Nul2TopEnriched}
The grading on $\Pi_{(1,2)}(X)$ defined in \ref{def:2TrackGpd} corresponds to the dimension of the cubes. For $x \in \Pi_{(1,2)}(X)$, we have $\deg(x) = 0$ if $x$ is a point in $X$, $\deg(x) = 1$ if $x$ is a left path in $X$, and $\deg(x) = 2$ if $x$ is a left $2$-track in $X$.  This $2$-graded set is the left $2$-cubical set $\Nul_2(X)$ \cite{Baues15}*{Definition 1.9}.
\end{rem}

\begin{defn} \label{def:HomotGpsDGpd}
Given a $2$-track groupoid $G$, its \Def{homotopy groups} are
\begin{align*}
&\pi_0 G = \Comp G_{(1)} \\
&\pi_1 G = \Aut_{G_{(1)}}(0) \\
&\pi_2 G = \Aut_{G_{(2)}}(0).
\end{align*}
Note that $\pi_0 G$ is a priori only a pointed set, $\pi_1 G$ is a group, and $\pi_2 G$ is an abelian group.

A morphism $F \colon G \to G'$ of $2$-track groupoids is a \Def{weak equivalence} if it induces an isomorphism on homotopy groups.
\end{defn}

\begin{rem} \label{rem:HomotGps}
Let $X$ be a topological space with basepoint $x_0 \in X$. Then the homotopy groups of its fundamental $2$-track groupoid $G = \Pi_{(1,2)}(X,x_0)$ are the homotopy groups of the space $\pi_i G = \pi_i(X,x_0)$ for $i = 0,1,2$.
\end{rem}

The following two lemmas are straightforward.

\begin{lem} \label{lem:ProdDGpd}
$\DGpd$ has products, given by $G \x G' = \left( G_{(1)} \x G'_{(1)}, G_{(2)} \x G'_{(2)} \right)$, and where the structural isomorphisms
\[
\psi_{(x,x')} \colon \Aut_{G_{(2)} \x G'_{(2)}} \left( (x,x') \right) \ral{\simeq} \Aut_{G_{(2)} \x G'_{(2)}} \left( (0,0') \right)
\]
are given by $\psi_x \x \psi_{x'}$, and the quotient function
\[
\xymatrix{
**[r] (G \x G')_{(2)0} = G_{(2)0} \x G'_{(2)0} \ar@{->>}[d]^{q \x q'} \\
**[r] \Star (G \x G')_{(1)} = \Star G_{(1)} \x \Star G'_{(1)}
}
\]
is the product of the quotient functions for $G$ and $G'$.
\end{lem}

\begin{lem} \label{lem:FundPreserveProd}
The fundamental $2$-track groupoid preserves products:
\[
\Pi_{(1,2)}(X \x Y) \cong \Pi_{(1,2)}(X) \x \Pi_{(1,2)}(Y).
\]
\end{lem}

\section{$2$-tracks in a topologically enriched category}

Throughout this section, let $\cat{C}$ be a category enriched in $(\Topp_*,\sm)$. Explicitly:
\begin{itemize}
\item For any objects $A$ and $B$ of $\cat{C}$, there is a morphism space $\cat{C}(A,B)$ with basepoint denoted $0 \in \cat{C}(A,B)$.
\item For any objects $A$, $B$, and $C$, there is a composition map
\[
\mu \colon \cat{C}(B,C) \x \cat{C}(A,B) \to \cat{C}(A,C)
\]
which is associative and unital.
\item Composition satisfies
\[
\mu(x,0) = \mu(0,y) = 0
\]
for all $x$ and $y$. 
\end{itemize}
We write $x \in \cat{C}$ if $x \in \cat{C}(A,B)$ for some objects $A$ and $B$. For $x,y \in \cat{C}$, we write $xy = \mu(x,y)$ when $x$ and $y$ are composable, i.e., when the target of $y$ is the source of $x$. From now on, whenever an expression such as $xy$ or $x \ot y$ appears, it is understood that $x$ and $y$ must be composable.

By Definition \ref{def:ProdCubes}, we have the $\ot$-composition $x \ot y$ for $x,y \in \Pi_{(1)} \cat{C}$ and $\deg(x) + \deg(y) \leq 1$. For $\deg(a) = \deg(b) = 1$, we have:
\begin{align*}
ab &= \left( a \ot \de_1 b \right) \sq \left( \de_0 a \ot b \right) \\
&= \left( \de_1 a \ot b \right) \sq \left( a \ot \de_0 b \right).
\end{align*}
This equation holds in any category enriched in groupoids, where $ab$ denotes the (pointwise) composition. Note that for paths $\tild{a}$ and $\tild{b}$ representing $a$ and $b$, the boundary of the $2$-cube $\tild{a} \ot \tild{b}$ corresponds to the equation.

Conversely, the $\ot$-composition in $\Pi_{(1)} \cat{C}$ is determined by the pointwise composition. For $\deg(x) = \deg(y) = 0$ and $\deg(a) = 1$, we have:
\begin{equation} \label{eq:TensorFromPointwise}
\begin{cases}
x \ot y = xy \\
x \ot a = \id^{\sq}_x a \\
a \ot x = a \id^{\sq}_x. \\
\end{cases}
\end{equation}

We now consider the $2$-track groupoids $\Pi_{(1,2)} \cat{C}(A,B)$ of morphism spaces in $\cat{C}$, and we write $x \in \Pi_{(1,2)} \cat{C}$ if $x \in \Pi_{(1,2)} \cat{C}(A,B)$ for some objects $A$, $B$ of $\cat{C}$. By Definition \ref{def:ProdCubes}, composition in $\cat{C}$ induces the $\ot$-composition:
\[
x \ot y \in \Pi_{(1,2)} \cat{C}
\]
if $x$ and $y$ satisfy $\deg(x) + \deg(y) \leq 2$. For $\deg(x) = \deg(y) = 1$, $x$ and $y$ are left paths, hence $x \ot y$ is well-defined. The $\ot$-composition satisfies:
\[
\deg(x \ot y) = \deg(x) + \deg(y).
\]

The $\ot$-composition is associative, since composition in $\cat{C}$ is associative. The identity elements $1_A \in \cat{C}(A,A)$ for $\cat{C}$ provide identity elements $1 = 1_A \in \Pi_{(1,2)} \cat{C}(A,A)$, with $\deg(1_A) = 0$, and $x \ot 1 = x = 1 \ot x$.

Let us describe the $\ot$-composition of left paths more explicitly. Given left paths $a$ and $b$, then $a \ot b$ is a $2$-track from $\de_0(a \ot b) = (\de a) \ot b$ to $\de_1(a \ot b) = a \ot (\de b)$, as illustrated here:
\[
\xymatrix @C=5pc @R=5pc {
\ar@{-}[d]_-{\de_0(a \ot b) = \de a \ot b} \ar@{-}[r]^0 & \ar@{-}[d]^0 \\
\ar@{-}[r]_-{\de_1(a \ot b) = a \ot \de b} \urtwocell<\omit>{\quad a \ot b} & \\
}
\]

\begin{defn}
The \Def{$2$-track algebra associated to $\cat{C}$}, denoted $\left( \Pi_{(1)} \cat{C}, \Pi_{(1,2)} \cat{C}, \sq, \ot \right)$, consists of the following data.
\begin{itemize}
\item $\Pi_{(1)} \cat{C}$ is the category enriched in pointed groupoids given by the fundamental groupoids $\left( \Pi_{(1)} \cat{C}(A,B), \sq \right)$ of morphism spaces in $\cat{C}$, along with the $\ot$-composition, which determines (and is determined by) the composition in $\Pi_{(1)} \cat{C}$. 
\item $\Pi_{(1,2)} \cat{C}$ is given by the collection of fundamental $2$-track groupoids $\left( \Pi_{(1,2)} \cat{C}(A,B), \sq \right)$ together with the $\ot$-composition $x \ot y$ for $x,y \in \Pi_{(1,2)} \cat{C}$ satisfying $\deg(x) + \deg(y) \leq 2$.
\end{itemize}
\end{defn}

\begin{prop} \label{2TrackEndpoint}
Let $x, \al, \be \in \Pi_{(1,2)} \cat{C}$ with $\deg(x) = 0$ and $\deg(\al) = \deg(\be) = 2$. Then the following equations hold:\[
\begin{cases}
x \ot (\be \sq \al) = (x \ot \be) \sq (x \ot \al) \\
(\be \sq \al) \ot x = (\be \ot x) \sq (\al \ot x). \\
\end{cases}
\]
\end{prop}

\begin{proof}
This follows from functoriality of $\Pi_{(2)}$ applied to the composition maps $\mu(x,-) \colon \cat{C}(A,B) \to \cat{C}(A,C)$ and $\mu(-,x) \colon \cat{C}(B,C) \to \cat{C}(A,C)$.
\end{proof}

\begin{prop} \label{Boundary3Cube}
Let $c, \al \in \Pi_{(1,2)} \cat{C}$ with $\deg(c) = 1$ and $\deg(\al) = 2$. Then the following equations hold:\[
\begin{cases}
\de_1 \al \ot c = \left( \al \ot \de c \right) \sq \left( \de_0 \al \ot c \right) \\
c \ot \de_0 \al = \left( c \ot \de_1 \al \right) \sq \left( \de c \ot \al \right). \\
\end{cases}
\]
\end{prop}

\begin{proof}
Write $a = \de_0 \al$ and $b = \de_1 \al$, i.e., $\al$ is a left $2$-track from $a$ to $b$:
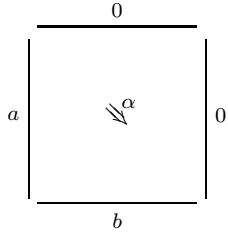
\[
\xymatrix  @C=5pc @R=5pc {
\ar@{-}[d]_{a} \ar@{-}[r]^0 & \ar@{-}[d]^0 \\
\ar@{-}[r]_{b} \urtwocell<\omit>{\al} & \\
}
\]
and note in particular $\de a = \de b$. Let $\tild{\al}$ be a left $2$-cube that represents $\al$ and consider the left $3$-cube $\tild{\al} \ot c$:
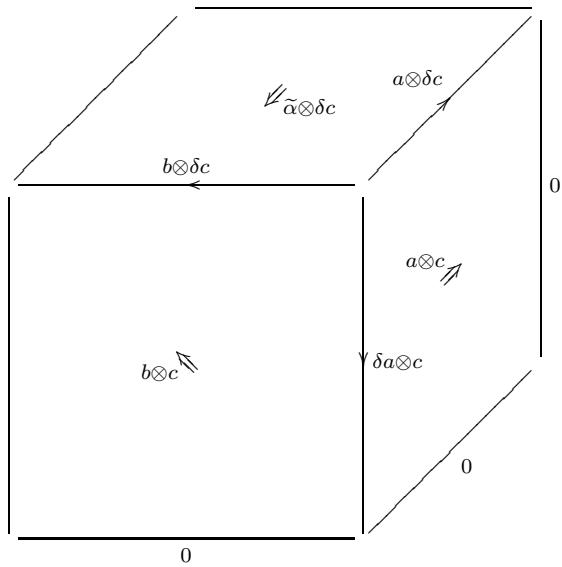
\[
\xymatrix  @C=5pc @R=5pc {
& \ar@{-}[rr] \drtwocell<\omit>{\hspace{1.5em} \tild{\al} \ot \de c} & & \ar@{-}[dd]^0 \\
\ar@{-}[dd] \ar@{-}[ur] & & \ar@{-}[ll]|-*=0@{>}_{b \ot \de c} \ar@{-}[dd]|-*=0@{>}^{\de a \ot c} \ar@{-}[ur]|-*=0@{>}^{a \ot \de c} \ddlltwocell<\omit>{\hspace{-1em} b \ot c} & \\
& & & \ultwocell<\omit>{\hspace{-1em} a \ot c} \\
\ar@{-}[rr]_0 & & \ar@{-}[ur]_0 & \\
}
\]
Its boundary exhibits the equality of $2$-tracks:
\begin{align*}
&\text{top face } \sq \text{ right face } = \text{ front face} \\
&\left( \al \ot \de c \right) \sq (a \ot c) = b \ot c \\
&\left( \al \ot \de c \right) \sq (\de_0 \al \ot c) = \de_1 \al \ot c.
\end{align*}

Likewise, for second equation, consider the left $3$-cube $c \ot \tild{\al}$:
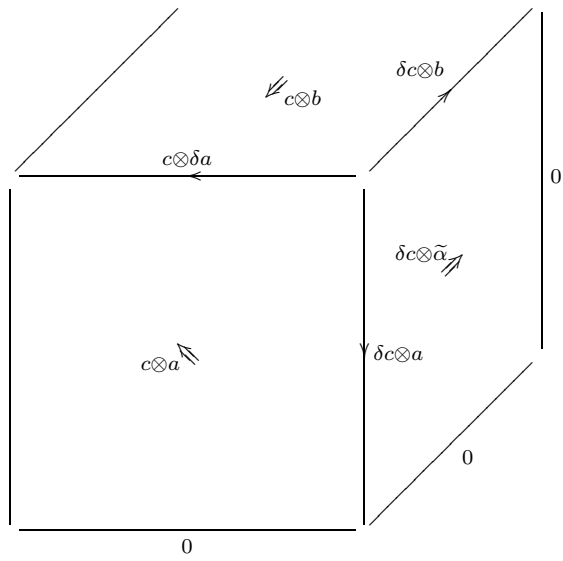
\[
\xymatrix  @C=5pc @R=5pc {
& \ar@{-}[rr] \drtwocell<\omit>{\hspace{1em} c \ot b} & & \ar@{-}[dd]^0 \\
\ar@{-}[dd] \ar@{-}[ur] & & \ar@{-}[ll]|-*=0@{>}_{c \ot \de a} \ar@{-}[dd]|-*=0@{>}^{\de c \ot a} \ar@{-}[ur]|-*=0@{>}^{\de c \ot b} \ddlltwocell<\omit>{\hspace{-1em} c \ot a} & \\
& & & \ultwocell<\omit>{\hspace{-1.3em} \de c \ot \tild{\al}} \\
\ar@{-}[rr]_0 & & \ar@{-}[ur]_0 & \\
}
\]
Its boundary exhibits the equality of $2$-tracks:
\begin{align*}
&\text{top face } \sq \text{ right face } = \text{ front face} \\
&\left( c \ot b \right) \sq \left( \de c \ot \al \right) = c \ot a \\
&\left( c \ot \de_1 \al \right) \sq \left( \de c \ot \al \right) = c \ot \de_0 \al. \qedhere
\end{align*}
\end{proof}

\section{$2$-track algebras}

We now collect the structure found in $\left( \Pi_{(1)} \cat{C}, \Pi_{(1,2)} \cat{C}, \sq, \ot \right)$ into the following definition.

\begin{defn} \label{def:2TrackAlg}
A \Def{$2$-track algebra} $\cat{A} = \left( \cat{A}_{(1)}, \cat{A}_{(1,2)}, \sq, \ot \right)$ consists of the following data.

\begin{enumerate}
\item A category $\cat{A}_{(1)}$ enriched in pointed groupoids, with the $\ot$-composition determined by Equation \eqref{eq:TensorFromPointwise}.

\item A collection $\cat{A}_{(1,2)}$ of $2$-track groupoids $\left( \cat{A}_{(1,2)}(A,B), \sq \right)$ for all objects $A,B$ of $\cat{A}_{(1)}$, such that the first groupoid in $\cat{A}_{(1,2)}(A,B)$ is equal to the pointed groupoid $\cat{A}_{(1)}(A,B)$.

\item For $x,y \in \cat{A}_{(1,2)}$, the $\ot$-composition $x \ot y \in \cat{A}_{(1,2)}$ is defined. For $\deg(x) = 0$ and $\deg(y) = 1$, the following equations hold in $\cat{A}_{(1)}$:
\[
\begin{cases}
q(x \ot y) = x \ot q(y) \\
q(y \ot x) = q(y) \ot x. \\
\end{cases}
\]
\end{enumerate}

The following equations are required to hold.

\begin{enumerate}
\item (\textit{Associativity}) $\ot$ is associative: $(x \ot y) \ot z = x \ot (y \ot z)$.

\item (\textit{Units}) The units $1 \in \cat{A}_{(1)}$, with $\deg(1_A) = 0$, serve as units for $\ot$, i.e., satisfy $x \ot 1 = x = 1 \ot x$ for all $x \in \cat{A}_{(1,2)}$.

\item (\textit{Pointedness}) $\ot$ satisfies $x \ot 0 = 0$ and $0 \ot y = 0$.

\item For $x, y, \al, \be \in \cat{A}_{(1,2)}$ with $\deg(x) = \deg(y) = 0$ and $\deg(\al) = \deg(\be) = 2$, we have:
\[
\begin{cases}
\de_i(x \ot \al \ot y) = x \ot (\de_i \al) \ot y \: \text{ for } i =0,1 \\
x \ot (\be \sq \al) \ot y = (x \ot \be \ot y) \sq (x \ot \al \ot y) \\
\end{cases}
\]
\item For $a,b \in \cat{A}_{(1,2)}$ with $\deg(a) = \deg(b) = 1$, we have:
\[
\begin{cases}
\de_0 (a \ot b) = \de a \ot b \\
\de_1 (a \ot b) = a \ot \de b. \\
\end{cases}
\] 

\item For $c, \al \in \cat{A}_{(1,2)}$ with $\deg(c) = 1$ and $\deg(\al) = 2$, we have:
\[
\begin{cases}
\de_1 \al \ot c = \left( \al \ot \de c \right) \sq \left( \de_0 \al \ot c \right) \\
c \ot \de_0 \al = \left( c \ot \de_1 \al \right) \sq \left( \de c \ot \al \right). \\
\end{cases}
\]
\end{enumerate}
\end{defn}

\begin{defn}
A \Def{morphism of $2$-track algebras} $F \colon \cat{A} \to \cat{B}$ consists of the following.
\begin{enumerate}
\item A functor $F_{(1)} \colon \cat{A}_{(1)} \to \cat{B}_{(1)}$ enriched in pointed groupoids.

\item A collection $F_{(1,2)}$ of morphisms of $2$-track groupoids
\[
F_{(1,2)}(A,B) \colon \cat{A}_{(1,2)}(A,B) \to \cat{B}_{(1,2)}(FA,FB)  
\]
for all objects $A,B$ of $\cat{A}$, such that $F_{(1,2)}(A,B)$ restricted to the first groupoid in $\cat{A}_{(1,2)}(A,B)$ is the functor $F_{(1)}(A,B) \colon \cat{A}_{(1)}(A,B)\to \cat{B}_{(1)}(FA,FB)$.

\item (\textit{Compatibility with $\ot$}) $F$ commutes with $\ot$:
\[
F(x \ot y) = Fx \ot Fy.
\]
\end{enumerate}
Denote by $\ttA$ the category of $2$-track algebras.
\end{defn}

\begin{defn}
Let $\cat{A}$ be a $2$-track algebra. The underlying \Def{homotopy category} of $\cat{A}$ is the homotopy category of the underlying track category $\cat{A}_{(1)}$, denoted
\[
\pi_0 \cat{A} := \pi_0 \cat{A}_{(1)} = \Comp \cat{A}_{(1)}.
\] 
We say that $\cat{A}$ is \Def{based} on the category $\pi_0 \cat{A}$.
\end{defn}

\begin{defn}
A morphism of $2$-track algebras $F \colon \cat{A} \to \cat{B}$ is a \Def{weak equivalence} (or \emph{Dwyer-Kan equivalence}) if the following conditions hold:
\begin{enumerate}
\item For all objects $A$ and $B$ of $\cat{A}$, the morphism
\[
F_{(1,2)} \colon \cat{A}_{(1,2)}(A,B) \to \cat{B}_{(1,2)}(FA,FB)
\]
is a weak equivalence of $2$-track groupoids (Definition \ref{def:HomotGpsDGpd}).
\item The induced functor $\pi_0 F \colon \pi_0 \cat{A} \to \pi_0 \cat{B}$ is an equivalence of categories.
\end{enumerate}
\end{defn}

\section{Higher order chain complexes}

In this section, we construct tertiary chain complexes, extending the work of \cite{BauesJ06} on secondary chain complexes. We will follow the treatment therein.

\begin{defn}
A \Def{chain complex} $(A,d)$ in a pointed category $\ct{A}$ is a sequence of objects and morphisms
\[
\xymatrix{
\cdots \ar[r] & A_{n+1} \ar[r]^{d_{n}} & A_n \ar[r]^{d_{n-1}} & A_{n-1} \ar[r] & \cdots \\ 
}
\]
in $\ct{A}$ satisfying $d_{n-1} d_{n} = 0$ for all $n \in \Z$. The map $d$ is called the \emph{differential}.
 
A \Def{chain map} $f \colon (A,d) \to (A',d')$ between chain complexes is a sequence of morphisms $f_n \colon A_n \to A'_n$ commuting with the differentials:
\[
\xymatrix{
\cdots \ar[r] & A_{n+1} \ar[d]_{f_{n+1}} \ar[r]^{d_{n}} & A_n \ar[d]_{f_{n}} \ar[r]^{d_{n-1}} & A_{n-1} \ar[d]_{f_{n-1}} \ar[r] & \cdots \\ 
\cdots \ar[r] & A'_{n+1} \ar[r]^{d'_{n}} & A'_n \ar[r]^{d'_{n-1}} & A'_{n-1} \ar[r] & \cdots \\ 
}
\]
i.e., satisfying $f_n d_{n} = d'_{n} f_{n+1}$ for all $n \in \Z$.
\end{defn}

\begin{defn}
\cite{BauesJ06}*{Definition 2.6} Let $\ct{B}$ be a category enriched in pointed groupoids. A \Def{secondary pre-chain complex} $(A,d,\ga)$ in $\ct{B}$ is a diagram of the form:
\[
\xymatrix{
\cdots \ar[r] \rrlowertwocell<-12>_0 & A_{n+2} \ar[r]^{d_{n+1}} \rruppertwocell<12>^0{^\ga_n} & A_{n+1} \ar[r]^{d_{n}} \rrlowertwocell<-12>_0{\quad\ga_{n-1}} & A_{n} \ar[r]^{d_{n-1}} \rruppertwocell<12>^0{^} & A_{n-1} \ar[r] & \cdots \\
}
\]
More precisely, the data consists of a sequence of objects $A_n$ and maps $d_{n} \colon A_{n+1} \to A_n$, together with left tracks $\ga_n \colon d_n d_{n+1} \Ra 0$ for all $n \in \Z$.

$(A,d,\ga)$ is a \Def{secondary chain complex} if moreover for each $n \in \Z$, the tracks
\[
\xymatrix @C=3pc {
d_{n-1} d_n d_{n+1} \ar@{=>}[r]^-{d_{n-1} \ot \ga_n} & d_{n-1} 0 \ar@{=>}[r]^-{\id_0^{\sq}} & 0 \\
}
\]
and
\[
\xymatrix @C=3pc {
d_{n-1} d_n d_{n+1} \ar@{=>}[r]^-{\ga_{n-1} \ot d_{n+1}} & 0 d_{n+1} \ar@{=>}[r]^-{\id_0^{\sq}} & 0 \\
}
\]
coincide. In other words, the track
\[
\cat{O}(\ga_{n-1},\ga_n) := \left( \ga_{n-1} \ot d_{n+1} \right) \sq \left( d_{n-1} \ot \ga_n \right)^{\inv} \colon 0 \Ra 0
\]
in the groupoid $\ct{B}(A_{n+2},A_{n-1})$ is the identity track of $0$.

We say that the secondary pre-chain complex $(A,d,\ga)$ is \Def{based} on the chain complex $(A,\{d\})$ in the homotopy category $\pi_0 \ct{B}$. 
\end{defn}

\begin{rem}
One can show that the notion of secondary (pre-)chain complex in $\ct{B}$ coincides with the notion of \emph{$1^{\text{st}}$ order (pre-)chain complex} in $\Nul_1 \ct{B}$ described in \cite{Baues15}*{\S 4, c.f. Example 12.3}.
\end{rem}

\begin{defn}
A \Def{tertiary pre-chain complex} $(A,d,\de,\xi)$ in a $2$-track algebra $\cat{A}$ is a sequence of objects $A_n$ and maps $d_{n} \colon A_{n+1} \to A_n$ in the category $\cat{A}_{(1)0}$, together with left paths $\ga_n \colon d_n d_{n+1} \to 0$ in $\cat{A}_{(1,2)}$, as illustrated in the diagram
\newcommand{\shortar}[1]{\ar@{}[#1]^(.25){}="a"^(.75){}="b" \ar "a";"b"}
\newcommand{\shortdar}[1]{\ar@{}[#1]^(.25){}="a"^(.75){}="b" \ar@{=>} "a";"b"} \[
\xymatrix @C=3pc {
& & & & & & \\
\cdots \ar[r] \ar@/^3pc/[rr]^0 & A_{n+3} \shortar{u} \ar[r]|{d_{n+2}} \ar@/_3pc/[rr]_0 & A_{n+2} \shortar{d}^{\ga_{n+1}} \ar[r]|{d_{n+1}} \ar@/^3pc/[rr]^0 & A_{n+1} \shortar{u}_{\ga_n} \ar[r]|{d_{n}} \ar@/_3pc/[rr]_0 & A_{n} \shortar{d}^{\ga_{n-1}} \ar[r]|{d_{n-1}} \ar@/^3pc/[rr]^0 & A_{n-1} \shortar{u} \ar[r] & \cdots \\
& & & & & & \\
}
\]
along with left $2$-tracks $\xi_n \colon \ga_n \ot d_{n+2} \Ra d_n \ot \ga_{n+1}$ in $\cat{A}_{(1,2)}$, for all $n \in \Z$.

$(A,d,\ga,\xi)$ is a \Def{tertiary chain complex} if moreover for each $n \in \Z$, the left $2$-track:
\[
\xymatrix @C=3pc {
d_{n-1} \ot \ga_n \ot d_{n+2} \ar@{=>}[r]^-{d_{n-1} \ot \xi_n} & d_{n-1} d_n \ot \ga_{n+1} \ar@{=>}[r]^-{\ga_{n-1} \ot \ga_{n+1}} & \ga_{n-1} \ot d_{n+1} d_{n+2} \ar@{=>}[r]^-{\xi_{n-1} \ot d_{n+2}} & d_{n-1} \ot \ga_n \ot d_{n+2} \\
}
\]
is the identity of $d_{n-1} \ot \ga_n \ot d_{n+2}$ in the groupoid $\cat{A}_{(2)}(A_{n+3},A_{n-1})$. In other words, the element:
\begin{equation*}
\begin{split}
\cat{O}(\xi_{n-1},\xi_n) := \psi_{d_{n-1} \ot \ga_n \ot d_{n+2}} \left( \left( \xi_{n-1} \ot d_{n+2} \right) \sq \left( \ga_{n-1} \ot \ga_{n+1} \right) \sq \left( d_{n-1} \ot \xi_n \right) \right) \\
\in \pi_2 \cat{A}_{(1,2)}(A_{n+3},A_{n-1})
\end{split}
\end{equation*}
is trivial. Here, $\psi$ is the structural isomorphism in the $2$-track groupoid $\cat{A}_{(1,2)}(A_{n+3},A_{n-1})$, as in Definitions \ref{def:StrictlyAbel} and \ref{def:2TrackGpd}.

We say that the tertiary pre-chain complex $(A,d,\ga,\xi)$ is \Def{based} on the chain complex $(A,\{d\})$ in the homotopy category $\pi_0 \cat{A}$.
\end{defn}

\subsection*{Toda brackets of length 3 and 4}

Let $\CC$ be a category enriched in $(\Topp_*, \sm)$. Let $\pi_0 \CC$ be the category of path components of $\CC$ (applied to each mapping space) and let
\[
\xymatrix{
Y_0 & Y_1 \ar[l]_{y_1} & Y_2 \ar[l]_{y_2} & Y_3 \ar[l]_{y_3} & Y_4 \ar[l]_{y_4} \\
}
\]
be a diagram in $\pi_0 \CC$ satisfying $y_1 y_2 = 0$, $y_2 y_3 = 0$, and $y_3 y_4 = 0$. Choose maps $x_i$ in $\CC$ representing $y_i$. Then there exist left $1$-cubes $a$, $b$, $c$ as in the diagram
\[
\arrowobject{\dir{->}}
\xymatrix{
Y_0 & Y_1 \ar[l]^{x_1} & Y_2 \ar[l]^{x_2} \lllowertwocell<-10>_0{a} & Y_3 \ar[l]^{x_3} \lllowertwocell<10>^0{^b} & Y_4. \ar[l]^{x_4} \lllowertwocell<-10>_0{c} \\
}
\]

\begin{defn}
The \Def{Toda bracket} of length 3, denoted $\lan y_1, y_2, y_3 \ran \subseteq \pi_1 \CC(Y_3,Y_0)$, is the set of all elements in $\Aut(0) = \pi_1 \CC(Y_3,Y_0)$ of the form
\[
\cat{O}(a,b) := (a \ot x_3) \sq (x_1 \ot b)^{\inv}
\]
as above.

Assume now that we can choose left $2$-tracks $\al \colon a \ot x_3 \Ra x_1 \ot b$ and $\be \colon b \ot x_4 \Ra x_2 \ot c$ in $\Pi_{(1,2)} \cat{C}$. Then the composite of left $2$-tracks
\[
(\al \ot x_4) \sq (a \ot c) \sq (x_1 \ot \be)
\]
is an element of $\Aut(x_1 \ot b \ot x_4)$, to which we apply the structural isomorphism
\[
\psi_{x_1 \ot b \ot x_4} \colon \Aut(x_1 \ot b \ot x_4) \ral{\cong} \pi_2 \CC(Y_4,Y_0).
\]
The set of all such elements is the \Def{Toda bracket} of length 4, denoted $\lan y_1, y_2, y_3, y_4 \ran \subseteq \pi_2 \CC(Y_4,Y_0)$.

Note that the existence of $\al$, resp. $\be$, implies that the bracket $\lan y_1, y_2, y_3 \ran$, resp. $\lan y_2, y_3, y_4 \ran$ contains the zero element.
\end{defn}

\begin{rem}
For a secondary pre-chain complex $(A,d,\ga)$, we have
\[
\cat{O}(\ga_{n-1},\ga_n) \in \lan d_{n-1}, d_n, d_{n+1} \ran
\]
for every $n \in \Z$. Likewise, for a tertiary pre-chain complex $(A,d,\ga,\xi)$, we have
\[
\cat{O}(\xi_{n-1},\xi_n) \in \lan d_{n-1}, d_n, d_{n+1}, d_{n+2} \ran
\]
for every $n \in \Z$. However, the vanishing of these Toda brackets does not guarantee the existence of a tertiary chain complex based on the chain complex $(A,\{d\})$. In a secondary chain complex $(A,d,\ga)$, these Toda brackets vanish in a \emph{compatible} way, that is, the equations $\cat{O}(\ga_{n-1},\ga_n) = 0$ and $\cat{O}(\ga_{n},\ga_{n+1}) = 0$ involve the same left track $\ga_n \colon d_n d_{n+1} \Ra 0$.
\end{rem}

\section{The Adams differential $d_3$}

Let $\Spec$ denote the topologically enriched category of spectra and mapping spaces between them.  More precisely, start from a simplicial (or topological) model category of spectra, like that of Bousfield--Friedlander \cite{Bousfield78}*{\S 2}, or symmetric spectra or orthogonal spectra \cite{MandellM01}, and take $\Spec$ to be the full subcategory of fibrant-cofibrant objects; c.f. \cite{Baues15}*{Example 7.3}.

Let $H := H\F_p$ be the Eilenberg-MacLane spectrum for the prime $p$ and let $\steen = H^* H$ denote the mod $p$ Steenrod algebra. Consider the collection $\ct{EM}$ of all mod $p$ generalized Eilenberg-MacLane spectra that are bounded below and of finite type, i.e., degreewise finite products $A = \prod_i \Si^{n_i} H$ with $n_i \in \Z$ and $n_i \geq N$ for some integer $N$ for all $i$. Since the product is degreewise finite, the natural map $\bigvee_i \Si^{n_i} H \to \prod_i \Si^{n_i} H$ is an equivalence, so that the mod $p$ cohomology $H^* A$ is a free $\steen$-module. Moreover, the cohomology functor restricted to the full subcategory of $\Spec$ with objects $\ct{EM}$ yields an equivalence of categories in the diagram:
\[
\xymatrix{
\pi_0 \Spec^{\opp} \ar[r]^-{H^*} & \Mod_{\steen} \\
\pi_0 \ct{EM}^{\opp} \ar@{^{(}->}[u] \ar[r]^-{H^*}_{\cong} & \finMod_{\steen} \ar@{^{(}->}[u] \\
}
\]
where $\finMod_{\steen}$ denotes the full subcategory consisting of free $\steen$-modules which are bounded below and of finite type.

Given spectra $Y$ and $X$, consider the Adams spectral sequence:
\[
E_2^{s,t} = \Ext_{\steen}^{s,t} \left( H^*X, H^*Y \right) \Ra \left[ \Si^{t-s} Y, X^{\wedge}_p \right].
\]
Assume that $Y$ is a finite spectrum and $X$ is a connective spectrum of finite type, i.e., $X$ is equivalent to a CW-spectrum with finitely many cells in each dimension and no cells below a certain dimension. Then the mod $p$ cohomology $H^* X$ is an $\steen$-module which is bounded below and degreewise finitely generated (as an $\steen$-module, or equivalently, as an $\F_p$-vector space). Choose a free resolution of $H^* X$ as an $\steen$-module:
\[
\xymatrix{
\cdots \ar[r] & F_2 \ar[r]^-{e_1} & F_1 \ar[r]^-{e_0} & F_0 \ar[r]^-{\la} & H^* X \\
}
\]
where each $F_i$ is a free $\steen$-module of finite type and bounded below. This diagram can be realized as the cohomology of a diagram in the stable homotopy category $\pi_0 \Spec$:
\[
\xymatrix{
\cdots & A_2 \ar[l] & A_1 \ar[l]_-{d_1} & A_0 \ar[l]_-{d_0} & A_{-1} = X \ar[l]_-{\ep} \\
}
\]
with each $A_i$ in $\ct{EM}$ (for $i \geq 0$) and satisfying $H^* A_i \cong F_i$. We consider this diagram as a diagram in the opposite category $\pi_0 \Spec^{\opp}$ of the form:
\[
\xymatrix{
\cdots \ar[r] & A_2 \ar[r]^-{d_1} & A_1 \ar[r]^-{d_0} & A_0 \ar[r]^-{\ep} & A_{-1} = X \\
}
\]
Since $A_{\bu} \to X$ is an $\ct{EM}$-resolution of $X$ in $\pi_0 \Spec^{\opp}$, there exists a tertiary chain complex $(A,d,\ga,\xi)$ in $\Pi_{(1,2)} \Spec^{\opp}$ based on the resolution $A_{\bu}\to X$, by Theorem~\ref{thm:Resolution}.

\begin{nota}
Given spectra $X$ and $Y$, let $\ct{EM}\{X,Y\}$ denote the topologically enriched subcategory of $\Spec$ consisting of all spectra in $\ct{EM}$ and mapping spaces between them, along with the objects $X$ and $Y$, with the mapping spaces $\Spec(X,A)$ and $\Spec(Y,A)$ for all $A$ in $\ct{EM}$; c.f. \cite{BauesJ06}*{Remark 4.3} \cite{Baues15}*{Remark 7.5}. We consider the $2$-track algebra $\Pi_{(1,2)} \ct{EM}\{X,Y\}^{\opp}$, or any $2$-track algebra $\cat{A}$ weakly equivalent to it. In the following construction, everything will take place within $\Pi_{(1,2)} \ct{EM}\{X,Y\}^{\opp}$, but we will write $\Pi_{(1,2)} \Spec^{\opp}$ for notational convenience.
\end{nota}

Start with a class in the $E_2$-term:
\[
x \in E_2^{s,t} = \Ext_{\steen}^{s,t}(H^*X, H^*Y) = \Ext_{\steen}^{s,0}(H^*X, \Si^t H^*Y)
\]
represented by a cocycle $x' \colon F_s \to \Si^t H^*Y$, i.e., a map of $\steen$-modules satisfying $x' d_s = 0$. Realize $x'$ as the cohomology of a map $x'' \colon A_s \to \Si^t Y$ in $\Spec^{\opp}$. The equation $x' d_s = 0$ means that $x'' d_s$ is null-homotopic; let $\ga \colon x'' d_s \to 0$ be a null-homotopy. Consider the diagram in $\Spec^{\opp}$:
\[
\xymatrix{
\cdots \ar[r] & A_{s+2} \ar[r]^-{d_{s+1}} & A_{s+1} \ar[r]^-{d_{s}} & A_s \ar[d]^-{x''} \ar[r]^-{d_{s-1}} & A_{s-1} \ar[r] & \cdots \ar[r] & A_0 \ar[r]^-{\ep} & X \\
& & & \Si^t Y & & & \\
}
\]
Now consider the underlying secondary pre-chain complex in $\Pi_{(1)} \Spec^{\opp}$:
\begin{equation} \label{eq:SecondaryResol}
\xymatrix{
\cdots \ar[r] \rruppertwocell<12>^0{\omit} & A_{s+3} \ar[r]^-{d_{s+2}} \rrlowertwocell<-12>_0{\quad\ga_{s+1}} & A_{s+2} \ar[r]^-{d_{s+1}} \rruppertwocell<12>^0{^\ga_s} & A_{s+1} \ar[r]^-{d_{s}} \rrlowertwocell<-12>_0{\ga} & A_s \ar[r]^-{x''} & \Si^t Y \\
}
\end{equation}
in which the obstructions $\cat{O}(\ga_i,\ga_{i+1})$ are trivial, for $i \geq s$.

\begin{thm}
The obstruction $\cat{O}(\ga,\ga_s) \in \pi_1 \Spec^{\opp}(A_{s+2},\Si^t Y) = \pi_0 \Spec^{\opp}(A_{s+2},\Si^{t+1} Y)$ is a (co)cycle and does not depend on the choices, up to (co)boundaries, and thus defines an element:
\[
d_{(2)}(x) \in \Ext^{s+2,t+1}_{\steen}(H^* X, H^* Y).
\]
Moreover, this function
\[
d_{(2)} \colon \Ext^{s,t}_{\steen}(H^*X, H^*Y) \to \Ext^{s+2,t+1}_{\steen}(H^* X, H^* Y)
\]
is the Adams differential $d_2$.
\end{thm}

\begin{proof}
This is \cite{BauesJ06}*{Theorems 4.2 and 7.3}, or the case $n=1, m=3$ of \cite{Baues15}*{Theorem 15.11}.

Here we used the natural isomorphism:
\[
\Ext^{i,j}_{\pi_0 \ct{EM}^{\opp}}(H^* X, H^* Y) \cong \Ext^{i,j}_{\steen}(H^* X, H^* Y)
\]
where the left-hand side is defined as in Example \ref{ex:ExtGps}. Using the equivalence of categories $H^* \colon \pi_0 \ct{EM}^{\opp} \ral{\cong} \finMod_{\steen}$, this natural isomorphism follows from the natural isomorphisms:
\begin{align*}
\pi_0 \Spec^{\opp}(A_{s+2},\Si^{t+1} Y) &= \Hom_{\steen} \left( F_{s+2}, H^* \Si^{t+1} Y \right) \\
&= \Hom_{\steen} \left( F_{s+2}, \Si^{t+1} H^*Y \right).
\end{align*}
Cocycles modulo coboundaries in this group are precisely $\Ext^{s+2,t+1}_{\steen}(H^* X, H^* Y)$.
\end{proof}

Now assume that $d_2(x) = 0$ holds, so that $x$ survives to the $E_3$-term. Since the obstruction
\[
\cat{O}(\ga,\ga_s) = \left( \ga \ot d_{s+1} \right) \sq \left( x'' \ot \ga_s \right)^{\inv}
\]
vanishes, one can choose a left $2$-track $\xi \colon \ga \ot d_{s+1} \Ra x'' \ot \ga_s$, which makes \eqref{eq:SecondaryResol} into a tertiary pre-chain complex in $\Pi_{(1,2)} \Spec^{\opp}$. Since $(A,d,\ga,\xi)$ was a tertiary chain complex to begin with, the obstructions $\cat{O}(\xi_i,\xi_{i+1})$ are trivial, for $i \geq s$.

\begin{thm} \label{thm:d3ASS}
The obstruction $\cat{O}(\xi,\xi_s) \in \pi_2 \Spec^{\opp}(A_{s+3},\Si^t Y) = \pi_0 \Spec^{\opp}(A_{s+3},\Si^{t+2} Y)$ is a (co)cycle and does not depend on the choices up to (co)boundaries, and thus defines an element:
\[
d_{(3)}(x) \in E_3^{s+3,t+2}(X,Y).
\]
Moreover, this function
\[
d_{(3)} \colon E_3^{s,t}(X,Y) \to E_3^{s+3,t+2}(X,Y)
\]
is the Adams differential $d_3$.
\end{thm}

\begin{proof}
This is the case $n=2, m=4$ of \cite{Baues15}*{Theorem 15.11}. More precisely, by Theorem \ref{2TrackAlgBalls}, the tertiary chain complex $(A,d,\ga,\xi)$ in $\Pi_{(1,2)} \Spec^{\opp}$ yields a $2^{\text{nd}}$ order chain complex in $\Nul_{2} \Spec^{\opp}$ based on the same $\ct{EM}$-resolution $A_{\bu}\to X$ in $\pi_0 \Spec^{\opp}$. The construction of $d_{(3)}$ above corresponds to the construction $d_3$ in \cite{Baues15}*{Definition 15.8}.
\end{proof}

\begin{rem}
The groups $E_3^{s,t}(X,Y)$ are an instance of the secondary $\Ext$ groups defined in \cite{BauesJ06}*{\S 4}. Likewise, the next term $E_4^{s,t}(X,Y) = \ker d_{(3)} / \im d_{(3)}$ is a higher order $\Ext$ group as defined in \cite{Baues15}*{\S 15}.
\end{rem}

\begin{thm} \label{thm:WeakEqClass}
A weak equivalence of $2$-track algebras induces an isomorphism of higher $\Ext$ groups, compatible with the differential $d_{(3)}$. More precisely, let $F \colon \cat{A} \to \cat{A}'$ be a weak equivalence between $2$-track algebras $\cat{A}$ and $\cat{A}'$ which are weakly equivalent to $\Pi_{(1,2)} \ct{EM}\{X,Y\}^{\opp}$. Then $F$ induces isomorphisms $E_{3,\cat{A}}^{s,t}(X,Y) \ral{\cong} E_{3,\cat{A}'}^{s,t}(FX,FY)$ making the diagram
\[
\xymatrix{
E_{3,\cat{A}}^{s,t}(X,Y) \ar[d]_{\cong} \ar[r]^-{d_{(3),\cat{A}}} & E_{3,\cat{A}}^{s+3,t+2}(X,Y) \ar[d]_{\cong} \\
E_{3,\cat{A}'}^{s,t}(FX,FY)  \ar[r]^-{d_{(3),\cat{A}'}} & E_{3,\cat{A}'}^{s+3,t+2}(FX,FY) \\
}
\]
commute. Here the additional subscript $\cat{A}$ or $\cat{A}'$ denotes the ambient $2$-track category in which the secondary $\Ext$ groups and the differential are defined.
\end{thm}

\begin{proof}
This follows from the case $n=2$ of \cite{Baues15}*{Theorem 15.9}, or an adaptation of the proof of \cite{BauesJ06}*{Theorem 5.1}.\end{proof}

\section{Higher order resolutions}

In this section, we specialize some results of \cite{Baues15} about higher order resolutions to the case $n=2$. We use the fact that a $2$-track algebra has an underlying algebra of left $2$-cubical balls, which is the topic of Section \ref{sec:AlgLeftCub}.

First, we recall some background on \emph{relative} homological algebra; more details can be found in \cite{BauesJ06}*{\S 1}.

\begin{defn}
Let $\ct{A}$ be an additive category and $\ct{a} \subseteq \ct{A}$ a full additive subcategory.

\begin{enumerate}
\item A chain complex $(A,d)$ is \Def{$\ct{a}$-exact} if for every object $X$ of $\ct{a}$ the chain complex $\Hom_{\ct{A}}(X,A_{\bu})$ is an exact sequence of abelian groups.
\item A chain map $f \colon (A,d) \to (A',d')$ is an \Def{$\ct{a}$-equivalence} if for every object $X$ of $\ct{a}$, the chain map $\Hom_{\ct{A}}(X,f)$ is a quasi-isomorphism.

\item For an object $A$ of $\ct{A}$, an \Def{$A$-augmented chain complex} $A_{\bu}^{\ep}$ is a chain complex of the form
\[
\xymatrix{
\cdots \ar[r] & A_{1} \ar[r]^{d_0} & A_0 \ar[r]^{\ep} & A \ar[r] & 0 \ar[r] & \cdots \\ 
}
\]
i.e., with $A_{-1} = A$ and $A_n = 0$ for $n < -1$. Such a complex can be viewed as a chain map $\ep \colon A_{\bu} \to A$ where $A$ is a chain complex concentrated in degree $0$. The map $\ep = d_{-1}$ is called the \Def{augmentation}.

\item An \Def{$\ct{a}$-resolution} of $A$ is an $A$-augmented chain complex $A_{\bu}^{\ep}$ which is $\ct{a}$-exact and such that for all $n \geq 0$, the object $A_n$ belongs to $\ct{a}$. In other words, an $\ct{a}$-resolution of $A$ is a chain complex $A_{\bu}$ in $\ct{a}$ together with an $\ct{a}$-equivalence $\ep \colon A_{\bu} \to A$.
\end{enumerate}
\end{defn}

\begin{ex}
Consider the category $\ct{A} = \ct{Mod}_R$ of $R$-modules for some ring $R$, and the subcategory $\ct{a}$ of free (or projective) $R$-modules. This recovers the usual homological algebra of $R$-modules.
\end{ex}

\begin{defn}
Let $\cat{A}$ be an abelian category and $F \colon \ct{A} \to \cat{A}$ an additive functor. The \Def{$\ct{a}$-relative left derived functors} of $F$ are the functors $L^{\ct{a}}_n F \colon \ct{A} \to \cat{A}$ for $n \geq 0$ defined by
\[
(L^{\ct{a}}_n F) A = H_n \left( F(A_{\bu}) \right) 
\]
where $A_{\bu} \to A$ is any $\ct{a}$-resolution of $A$. 

Likewise, if $F \colon \ct{A}^{\opp} \to \cat{A}$ is a contravariant additive functor, its \Def{$\ct{a}$-relative right derived functors} of $F$ are defined by
\[
(R_{\ct{a}}^n F) A = H^n \left( F(A_{\bu}) \right). 
\]
\end{defn}

\begin{ex} \label{ex:ExtGps}
The $\ct{a}$-relative $\Ext$ groups are given by
\[
\Ext_{\ct{a}}^n(A,B) := \left( R_{\ct{a}}^n \Hom_{\ct{A}}(-,B) \right)(A) = H^n \Hom_{\ct{A}}(A_{\bu},B).
\]
\end{ex}

\begin{prop}[Correction of $1$-tracks]
Let $\ct{B}$ be a category enriched in pointed groupoids, such that its homotopy category $\pi_0 \ct{B}$ is additive. Let $\ct{a} \subseteq \pi_0 \ct{B}$ be a full additive subcategory. Let $(A,d,\ga)$ be a secondary pre-chain complex in $\ct{B}$ based on an $\ct{a}$-resolution $A_{\bu} \to X$ of an object $X$ in $\pi_0 \ct{B}$. Then there exists a secondary chain complex $(A,d,\ga')$ in $\ct{B}$ with the same objects $A_i$ and differentials $d_i$. In particular $(A,d,\ga')$ is also based on the $\ct{a}$-resolution $A_{\bu} \to X$.
\end{prop}

\begin{proof}
This follows from an adaptation of the proof of \cite{BauesJ06}*{Lemma 2.14}, or the case $n=1$ of \cite{Baues15}*{Theorem 13.2}.
\end{proof}

\begin{prop}[Correction of $2$-tracks]
Let $\cat{A}$ be a $2$-track algebra such that its homotopy category $\pi_0 \cat{A}$ is additive. Let $\ct{a} \subseteq \pi_0 \cat{A}$ be a full additive subcategory. Let $(A,d,\ga,\xi)$ be a tertiary pre-chain complex in $\cat{A}$ based on an $\ct{a}$-resolution $A_{\bu} \to X$ of an object $X$ in $\pi_0 \cat{A}$. Then there exists a tertiary chain complex $(A,d,\ga,\xi')$ in $\cat{A}$ with the same objects $A_i$, differentials $d_i$, and left paths $\ga_i$. In particular, $(A,d,\ga,\xi')$ is also based on the $\ct{a}$-resolution $A_{\bu} \to X$.
\end{prop}

\begin{proof}
This follows from the case $n=2$ of \cite{Baues15}*{Theorem 13.2}.
\end{proof}

\begin{thm}[Resolution Theorem] \label{thm:Resolution}
Let $\cat{A}$ be a $2$-track algebra such that its homotopy category $\pi_0 \cat{A}$ is additive. Let $\ct{a} \subseteq \pi_0 \cat{A}$ be a full additive subcategory. Let $A_{\bu} \to X$ be an $\ct{a}$-resolution in $\pi_0 \cat{A}$. Then there exists a tertiary chain complex in $\cat{A}$ based on the resolution $A_{\bu} \to X$.
\end{thm}

\begin{proof}
This follows from the resolution theorems \cite{Baues15}*{Theorems 8.2 and 14.5}.
\end{proof}

\section{Algebras of left $2$-cubical balls} \label{sec:AlgLeftCub}

\begin{prop} \label{LeftCubical}
Every left cubical ball of dimension $2$ is equivalent to $C_k$ for some $k \geq 2$, where $C_k = B_1 \cup \cdots \cup B_k$ is the left cubical ball of dimension $2$ consisting of $k$ closed $2$-cells going cyclically around the vertex $0$, with one common $1$-cell $e_i$ between successive $2$-cells $B_i$ and $B_{i+1}$, where by convention $B_{k+1} := B_1$.

See Figure \ref{fig:LeftCubical}, which is taken from \cite{Baues15}*{Figure 3}.
\end{prop}

\begin{proof}
Let $B$ be a left cubical ball of dimension $2$. For each closed $2$-cell $B_i$, equipped with its homeomorphism $h_i \colon I^2 \ral{\cong} B_i$, the faces $\del^1_1 B_i$ and $\del^1_2 B_i$ are required to be $1$-cells of the boundary $\del B \cong S^1$, while the faces $\del^0_1 B_i$ and $\del^0_2 B_i$ are not in $\del B$, and therefore must be faces of some other $2$-cells. In other words, we have $\del^0_1 B_i = \del^0_1 B_j$ or $\del^0_1 B_i = \del^0_2 B_j$ for some other $2$-cell $B_j$, in fact a unique $B_j$, because $B$ is homeomorphic to a $2$-disk.

Pick any $2$-cell of $B$ and call it $B_1$. Then the face $e_1 := \del^0_2 B_1$ appears as a face of exactly one other $2$-cell, which we call $B_2$. The remaining face $e_2$ of $B_2$ appears as a face of exactly one other $2$-cell, which we call $B_3$. Repeating this process, we list distinct $2$-cells $B_1, \ldots, B_k$, and $B_{k+1}$ is one of the previously labeled $2$-cells. Then $B_{k+1}$ must be $B_1$, with $e_k = \del^0_1 B_1$, since a $1$-cell cannot appear as a common face of three $2$-cells. Finally, this process exhausts all $2$-cells, because all $2$-cells share the common vertex $0$, which has a neighborhood homeomorphic to an open $2$-disk.
\end{proof}

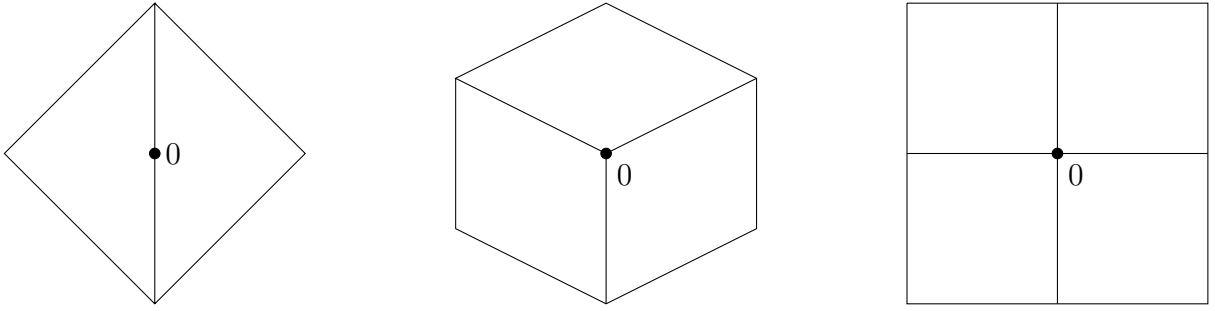
\begin{figure}[h]
\begin{tikzpicture}
\draw (0,-2) -- (0,2) -- (-2,0) -- (0,-2) -- (2,0) -- (0,2);
\draw[fill] (0,0) circle [radius=0.07];
\node [right] at (0,0) {0};
\draw (4,1) -- (4,-1) -- (6,-2) -- (8,-1) -- (8,1) -- (6,2) -- (4,1) -- (6,0) -- (6,-2);
\draw (6,0) -- (8,1);
\draw[fill] (6,0) circle [radius=0.07];
\node [below right] at (6,0) {0};
\draw (10,2) -- (10,-2) -- (14,-2) -- (14,2) -- (10,2);
\draw (10,0) -- (14,0);
\draw (12,-2) -- (12,2);
\draw[fill] (12,0) circle [radius=0.07];
\node [below right] at (12,0) {0};
\end{tikzpicture}
\caption{The left cubical balls $C_2$, $C_3$, and $C_4$.}
\label{fig:LeftCubical}
\end{figure}

\begin{prop} \label{Left2Cubical}
A left $2$-cubical ball (\cite{Baues15}*{Definition 10.1}) in a pointed space $X$ corresponds to a circular chain of composable left $2$-tracks:
\[
a = a_0 \ral{\al_1^{\ep_1}} a_1 \ral{\al_2^{\ep_2}} \cdots \to a_{k-1} \ral{\al_k^{\ep_k}} a_k = a
\]
where the sign $\ep_i = \pm 1$ is the orientation of the $2$-cells in the left cubical ball (\cite{Baues15}*{Definition 10.8}). Moreover, such an expression $(\al_1, \ldots, \al_k)$ of a left $2$-cubical ball is unique up to cyclic permutation of the $k$ left $2$-tracks $\al_i$. For example, $(\al_1, \al_2, \ldots, \al_k)$ and $(\al_2, \ldots, \al_k, \al_1)$ represent the same left $2$-cubical ball. See Figure \ref{fig:Left2Cubical}.
\end{prop}

\begin{proof}
By our convention for the $\sq$-composition, a left $2$-track $\al$ defines a morphism between left paths $\al \colon d^0_1 \al \Ra d^0_2 \al$. The gluing condition for a left $2$-cubical ball $(\al_1, \ldots, \al_k)$ based on a left cubical ball $B = B_1 \cup \cdots \cup B_k$ as in Proposition \ref{LeftCubical} is that the restrictions $\al_i \vert_{e_i}$ and $\al_{i+1} \vert_{e_i}$ agree on the common edge $e_i \subset B_i \cap B_{i+1}$. This is the composability condition for $\al_{i+1}^{\ep_{i+1}} \sq  \al_i^{\ep_i}$. Indeed, up to a global sign, the sign of $B_i$ is
\[
\ep_i = \begin{cases}
+1 &\text{if } e_i = \del^0_2 B_i \\
-1 &\text{if } e_i = \del^0_1 B_i \\
\end{cases}
\]
so that we have $\al_i^{\ep_i} \colon \al_i \vert_{e_{i-1}} \Ra \al_i \vert_{e_i}$ and we may take $a_i = \al_i \vert_{e_i}$.
\end{proof}

\begin{figure}[h]
\[
\xymatrix{
& a_k = a_0 \ar@{=>}[r]^-{\al_1^{\ep_1}} & a_1 \ar@{=>}[dr]^-{\al_2^{\ep_2}} & \\
a_{k-1} \ar@{=>}[ur]^-{\al_k^{\ep_k}} & & & a_2 \ar@{=>}[dl]^-{\al_3^{\ep_3}}  \\
& a_4 \ar@{=>}[ul]^{\cdots} & a_3 \ar@{=>}[l]^-{\al_4^{\ep_4}}  & \\
}
\]
\caption{A left $2$-cubical ball.}
\label{fig:Left2Cubical}
\end{figure}
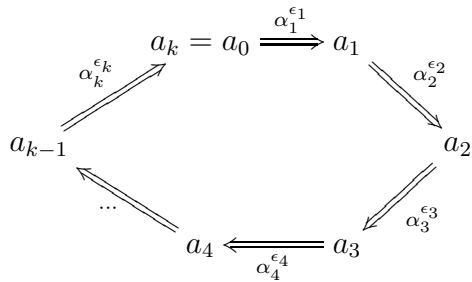

\begin{thm} \label{2TrackAlgBalls}
\begin{enumerate}
\item A $2$-track algebra $\cat{A}$ yields an algebra of left $2$-cubical balls (\cite{Baues15}*{Definition 11.1}) in the following way. Consider the system $\AlgCub (\cat{A}) := \left( (\cat{A}_{(1,2)},\ot), \pi_0 \cat{A}, D, \cat{O} \right)$, where:
\begin{itemize}
\item $(\cat{A}_{(1,2)},\ot)$ is the underlying $2$-graded category of $\TT$ (described in Definition \ref{def:2TrackAlg}).
\item $\pi_0 \cat{A}$ is the homotopy category of $\cat{A}$.\item $q \colon (\cat{A})^0 = \cat{A}_{(1)0} \surj \pi_0 \cat{A}$ is the canonical quotient functor.  
\item $D \colon \left( \pi_0 \cat{A} \right)^{\opp} \x \pi_0 \cat{A} \to \Ab$ is the functor defined by $D(A,B) = \pi_2 \cat{A}_{(1,2)}(A,B)$.
\item The obstruction operator $\cat{O}$ is obtained by concatenating the corresponding left $2$-tracks and using the structural isomorphisms $\psi$ of the mapping $2$-track groupoid:
\[
\OO_B(\al_1, \al_2, \ldots, \al_k) = \psi_{a} \left( \al_k^{\ep_k} \sq \cdots \sq \al_2^{\ep_2} \sq \al_1^{\ep_1} \right) \in \Aut_{\cat{A}_{(2)}(A,B)}(0) = \pi_2 \cat{A}_{(1,2)}(A,B)
\]
where we denoted $a = \de_0 \al_1 = \de_1 \al_k$.
\end{itemize}
\item Given a category $\cat{C}$ enriched in pointed spaces, $\AlgCub \left( \Pi_{(1,2)} \cat{C} \right)$ is the algebra of left $2$-cubical balls
\[
\left( \Nul_2 \cat{C}, \pi_0 \cat{C}, \pi_2 \cat{C}(-,-), \cat{O} \right)
\]
described in \cite{Baues15}*{\S 11}.\item The construction $\AlgCub$ sends a tertiary pre-chain complex $(A,d,\de,\xi)$ in $\cat{A}$ to a $2^{\text{nd}}$ order pre-chain complex in $\AlgCub(\cat{A})$, in the sense of \cite{Baues15}*{Definition 11.4}. Moreover, $(A,d,\de,\xi)$ is a tertiary chain complex if and only if the corresponding $2^{\text{nd}}$ order pre-chain complex in $\AlgCub(\cat{A})$ is a $2^{\text{nd}}$ order chain complex.
\end{enumerate}

\end{thm}

\begin{proof}
Let us check that the obstruction operator $\cat{O}$ is well-defined. By \ref{Left2Cubical}, the only ambiguity is the starting left $1$-cube $a_i$ in the composition. Two such compositions are conjugate in the groupoid $\cat{A}_{(2)}(A,B)$:
\begin{align*}
&\al_{i-1}^{\ep_{i-1}} \sq \cdots \sq \al_2^{\ep_2} \sq \al_1^{\ep_1} \sq \al_k^{\ep_k} \sq \cdots \sq \al_{i+1}^{\ep_{i+1}} \sq \al_i^{\ep_i} \\
= &\left( \al_{i-1}^{\ep_{i-1}} \sq \cdots \sq \al_1^{\ep_1} \right) \sq \al_k^{\ep_k} \sq \cdots \sq \al_{i+1}^{\ep_{i+1}} \sq \al_i^{\ep_i} \sq \cdots \sq \al_1^{\ep_1} \sq \left( \al_{i-1}^{\ep_{i-1}} \sq \cdots \sq \al_1^{\ep_1} \right)^{\inv} \\
= &\be^{\inv} \sq \al_k^{\ep_k} \sq \cdots \sq \al_1^{\ep_1} \sq \be
\end{align*}
with $\be = \left( \al_{i-1}^{\ep_{i-1}} \sq \cdots \sq \al_1^{\ep_1} \right)^{\inv} \colon a_i \Ra a_0$.
Since $\cat{A}_{(2)}(A,B)$ is a strictly abelian groupoid, we have the commutative diagram:
\[
\xymatrix{
\Aut_{}(a_0) \ar[dr]_{\psi_{a_0}} \ar[r]^-{\phy^{\be}} & \Aut_{}(a_i) \ar[d]^{\psi_{a_i}} \\
& \Aut(0) \\
}
\]
so that $\cat{O}_B(\al_1, \ldots, \al_k)$ is well-defined.

The remaining properties listed in \cite{Baues15}*{Definition 11.1} are straightforward verifications.
\end{proof}

\appendix

\section{Models for homotopy $2$-types} \label{sec:Models2Types}

Recall that the left $n$-cubical set $\Nul_n(X)$ of a pointed space $X$ depends only on the $n$-type $P_n X$ of $X$ \cite{Baues15}*{\S 1}. In particular the fundamental $2$-track groupoid $\Pi_{(1,2)}(X)$ depends only on the $2$-type $P_2 X$ of $X$. There are various algebraic models for homotopy $2$-types in the literature, using $2$-dimensional categorical structures. Let us mention the weak $2$-groupoids of \cite{Tamsamani99}, the bigroupoids of \cite{Hardie01}, the double groupoids of \cite{Brown02}, the two-typical double groupoids of \cite{Blanc11}, and the double groupoids with filling condition of \cite{Cegarra12}.

In contrast, $2$-track groupoids are \emph{not} models for homotopy $2$-types, not even of connected homotopy $2$-types. In the application we are pursuing, the functor $\Pi_{(1,2)}$ will be applied to topological abelian groups, hence products of Eilenberg-MacLane spaces. We are \emph{not} trying to encode the homotopy $2$-type of the Eilenberg-MacLane mapping theory, but rather as little information as needed in order to compute the Adams differential $d_3$.

The fundamental $2$-track groupoid $\Pi_{(1,2)}(X)$ encodes the $1$-type of $X$, via the fundamental groupoid $\Pi_{(1)}(X)$. Moreover, as noted in Remark~\ref{rem:HomotGps}, it also encodes the homotopy group $\pi_2(X)$. However, it fails to encode the $\pi_1(X)$-action on $\pi_2(X)$, as we will show below.

\subsection{Connected $2$-track groupoids}

Recall that a category $\cat{C}$ is called \emph{skeletal} if any isomorphic objects are equal. A \emph{skeleton} of $\CC$ is a full subcategory on a collection consisting of one representative object in each isomorphism class of objects of $\CC$. Every groupoid is equivalent to a disjoint union of groups, that is, a coproduct of single-object groupoids. The inclusion $\sk G \ral{\simeq} G$ of a skeleton of $G$ provides such an equivalence. A similar construction yields the following statement for $2$-track groupoids.

\begin{lem}
Let $G = \left( G_{(1)}, G_{(2)} \right)$ be a $2$-track groupoid.
\begin{enumerate}
\item There is a weak equivalence of $2$-track groupoids $\sk_{(1)} G \ral{\sim} G$ where the first groupoid of $\sk_{(1)} G$ is skeletal.
\item If $G$ is connected and $G_{(1)}$ is skeletal, then there is a weak equivalence of $2$-track groupoids $\sk_{(2)} G \ral{\sim} G$ where both groupoids of $\sk_{(2)} G$ are skeletal. 
\end{enumerate}
In particular, if $G$ is connected, then $\sk_{(2)} \sk_{(1)} G \ral{\sim} \sk_{(1)} G \ral{\sim} G$ is a weak equivalence between $G$ and a $2$-track groupoid whose constituent groupoids are both skeletal.
\end{lem}

\begin{lem} \label{lem:WeakEqDGpdSkel}
Let $G$ and $G'$ be connected $2$-track groupoids whose constituent groupoids are skeletal. If there are isomorphisms of homotopy groups $\phy_1 \colon \pi_1 G \simeq \pi_1 G'$ and $\phy_2 \colon \pi_2 G \simeq \pi_2 G'$, then there is a weak equivalence $\phy \colon G \ral{\sim} G'$.
\end{lem}

\begin{proof}
Since $G_{(1)}$ and ${G'_{(1)}}$ are skeletal, they are in fact groups, and the group isomorphism $\phy_1$ is an isomorphism of groupoids $\phy_{(1)} \colon G_{(1)} \ral{\simeq} G'_{(1)}$.

Now we define a functor $\phy_{(2)} \colon G_{(2)} \to G'_{(2)}$. On objects, it is given by the composite
\[
\xymatrix{
G_{(2)0} = \Comp G_{(2)} \ar[r]^-{q}_-{\simeq} & \Star G_{(1)} = G_{(1)}(0,0) = \pi_1 G \ar[r] & \\
\ar[r]^-{\phy_1}_-{\simeq} & \pi_1 G' = G'_{(1)}(0,0) = \Star G'_{(1)} & \Comp G'_{(2)} = G'_{(2)0} \ar[l]_-{q}^-{\simeq} \\  
}
\]
which is a bijection. On morphisms, $\phy_{(2)}$ is defined as follows. We have $G_{(2)}(a,b) = \emptyset$ when $a \neq b$, so there is nothing to define then. On the automorphisms of an object $a \in G_{(2)0}$, define $\phy_{(2)}$ as the composite
\[
\xymatrix{
G_{(2)}(a,a) = \Aut_{G_{(2)}}(a) \ar[r]^-{\psi_a}_-{\simeq} & \Aut_{G_{(2)}}(0) = \pi_2 G \ar[r] & \\
\ar[r]^-{\phy_2}_-{\simeq} & \pi_2 G' = \Aut_{G'_{(2)}}(0') & \Aut_{G'_{(2)}}(\phy(a)) = G'_{(2)}(\phy(a),\phy(a)). \ar[l]_-{\psi'_{\phy(a)}}^-{\simeq} \\
}
\]
Then $\phy_{(2)}$ is a functor and commutes with the structural isomorphisms, by construction. Thus $\phy = (\phy_{(1)}, \phy_{(2)}) \colon G \to G'$ is a morphism of $2$-track groupoids, and is moreover a weak equivalence.
\end{proof}

\begin{cor} \label{cor:WeakEqDGpd}
Let $G$ and $G'$ be connected $2$-track groupoids with isomorphic homotopy groups $\pi_i G \simeq \pi_i G'$ for $i = 1,2$. Then $G$ and $G'$ are weakly equivalent, i.e., there is a zigzag of weak equivalences between them.
\end{cor}

\begin{proof}
Consider the zigzag of weak equivalences
\[
\xymatrix{
G & G' \\
\sk_{(1)} G \ar[u]^{\sim} & \sk_{(1)} G' \ar[u]_{\sim} \\
\sk_{(2)} \sk_{(1)} G \ar[u]^{\sim} \ar[r]^-{\phy}_-{\sim} & \sk_{(2)} \sk_{(1)} G' \ar[u]_{\sim} \\
}
\]
where the bottom morphism $\phy$ is obtained from Lemma~\ref{lem:WeakEqDGpdSkel}. 
\end{proof}

By Remark~\ref{rem:HomotGps}, the functor $\Pi_{(1,2)} \colon \Topp_* \to \DGpd$ induces a functor
\begin{equation} \label{eq:HoLDGpd}
\Pi_{(1,2)} \colon \Ho \left( \mathbf{connected } \, 2\mathbf{-Types} \right) \to \Ho \left( \DGpd \right) 
\end{equation}
where the left-hand side denotes the homotopy category of connected $2$-types (localized with respect to weak homotopy equivalences), and the right-hand side denotes the localization with respect to weak equivalences, as in Definition~\ref{def:HomotGpsDGpd}. 

\begin{prop}
The functor $\Pi_{(1,2)}$ in \eqref{eq:HoLDGpd} is not an equivalence of categories.
\end{prop}

\begin{proof}
Let $X$ and $Y$ be connected $2$-types with isomorphic homotopy groups $\pi_1$ and $\pi_2$, but distinct $\pi_1$-actions on $\pi_2$. Then $X$ and $Y$ are not weakly equivalent, but $\Pi_{(1,2)}(X)$ and $\Pi_{(1,2)}(Y)$ are weakly equivalent, by Corollary~\ref{cor:WeakEqDGpd}.
\end{proof}

\subsection{Comparison to bigroupoids}

Any algebraic model for (pointed) homotopy $2$-types has an underlying $2$-track groupoid. Using the globular description in Remark~\ref{rem:Globular}, the most direct comparison is to the bigroupoids of \cite{Hardie01}. A \emph{pointed} bigroupoid (resp. double groupoid) will mean one equipped with a chosen object, here denoted $x_0$ to emphasize that it is unrelated to the algebraic structure of the bigroupoid.

\begin{prop} \label{pr:Bigroupoid}
Let $\Pi_2^{\bigpd}(X)$ denote the homotopy bigroupoid of a space $X$ constructed in \cite{Hardie01}, where it was denoted $\Pi_2(X)$.
\begin{enumerate}
\item There is a forgetful functor $U$ from pointed bigroupoids to $2$-track groupoids.
\item For a pointed space $X$, there is a natural isomorphism of $2$-track groupoids $\Pi_{(1,2)} (X) \cong U \Pi_2^{\bigpd}(X)$.
\end{enumerate}
\end{prop}

\begin{proof}
Let $B$ be a bigroupoid. We construct a $2$-track groupoid $UB$ as follows. The first constituent groupoid of $UB$ is the underlying groupoid of $B$
\[
UB_{(1)} := \pi_0 B
\]
obtained by taking the components of each mapping groupoid $B(x,y)$. The second constituent groupoid of $UB$ is a coproduct of mapping groupoids
\[
UB_{(2)} := \coprod_{x \in \Ob(B)} B(x,x_0).
\]
The quotient function $q \colon UB_{(2)0} \to \Star UB_{(1)}$ is induced by the natural quotient maps $\Ob \left( B(x,x_0) \right) \surj \pi_0 B(x,x_0)$. To define the structural isomorphisms
\[
\psi_a \colon \Aut(a) \ral{\cong} \Aut(c_{x_0})
\]
for objects $a \in UB_{(2)0}$, which are $1$-morphisms to the basepoint $a \colon x \to x_0$, consider the diagram
\[
\xymatrix @C=5.0pc {
x \rrlowertwocell<-18>_{a}{<5>\la} \rtwocell<6>^a_a{\hspace{1em} \id^{\sq}_a} & x_0 \rtwocell<6>^{c_{x_0}}_{c_{x_0}}{\al} & x_0 \lllowertwocell<-18>_{a}{<5>\la} \\ 
}
\]
where $\la \colon c_{x_0} \bu a \Ra a$ is the \emph{left identity} coherence $2$-isomorphism, $\bu$ denotes composition of $1$-morphisms, and $c_{x_0}$ is the identity $1$-morphism of the object $x_0$. (We kept our notation $\sq$ for composition of $2$-morphisms.) The inverse $\psi_a^{-1} \colon \Aut(c_{x_0}) \to \Aut(a)$ is defined by going from top to bottom in the diagram, namely
\[
\psi_a^{-1}(\al) = \la \sq \left( \al \bu \id^{\sq}_a \right) \sq \la^{\inv}.
\]
One readily checks that $UB$ is a $2$-track groupoid, that this construction $U$ is functorial, and that $U \Pi_2^{\bigpd}(X)$ is naturally isomorphic to $\Pi_{(1,2)}(X)$ as $2$-track groupoids. 
\end{proof}

\subsection{Comparison to double groupoids}

The homotopy double groupoid $\rho_2^{\sq}(X)$ from \cite{Brown02} is a cubical construction. Following the terminology therein, \emph{double groupoid} will be shorthand for \emph{edge symmetric double groupoid with connection}.

Let us recall the geometric idea behind $\rho_2^{\sq}(X)$. A path $a \colon I \to X$ has an underlying \emph{semitrack} $\lan a \ran$, defined as its equivalence class with respect to \emph{thin} homotopy rel $\del I$. A semitrack $\lan a \ran$ in turn has an underlying track $\{ a \}$. A square $u \colon I^2 \to X$ has an underlying $2$-track $\{ u \}$. A $2$-track $\{ u \}$ in turn has an underlying equivalence class $\{ u \}_T$ with respect to \emph{cubically thin homotopy}, i.e., a homotopy whose restriction to the boundary $\del I^2$ is thin (not necessarily stationary). The homotopy double groupoid $\rho_2^{\sq}(X)$ encodes semitracks $\lan a \ran$ in $X$ and $2$-tracks $\{ u \}_T$ up to cubically thin homotopy.

\begin{prop} \label{pr:DoubleGpd}
Let $\rho^{\sq}_2(X)$ denote the homotopy double groupoid of a space $X$ constructed in \cite{Brown02}.
\begin{enumerate}
\item There is a forgetful functor $U$ from pointed double groupoids to $2$-track groupoids.
\item For a pointed space $X$, there is a natural weak equivalence of $2$-track groupoids $\Pi_{(1,2)}(X) \ral{\sim} U \rho^{\sq}_2(X)$.
\end{enumerate}
\end{prop}

\begin{proof}
We adopt the notation of \cite{Brown02}, including that compositions in a double groupoid are written in diagrammatic order, i.e., $a + b$ denotes the composition $x \ral{a} y \ral{b} z$. However, we keep our graphical convention for the two axes:

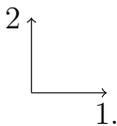
\begin{figure}[h]
\begin{tikzpicture}
\draw [->] (0,0) -- (0,1);
\draw [->] (0,0) -- (1,0);
\node [left] at (0,1) {$2$};
\node [below] at (1,0) {$1$.};
\end{tikzpicture}
\end{figure}

Let $D$ be a double groupoid, whose data is represented in the diagram of sets
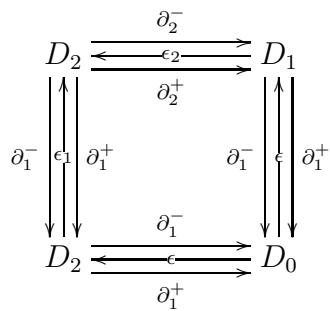
\[
\xymatrix @R=5pc @C=5pc {
D_2 \ar@<1ex>[d]^-{\del^{+}_1} \ar@<-1ex>[d]_-{\del^{-}_1} \ar@<1ex>[r]^-{\del^{-}_2} \ar@<-1ex>[r]_-{\del^{+}_2} & D_1 \ar[l]|-{\ep_2} \ar@<1ex>[d]^-{\del^{+}_1} \ar@<-1ex>[d]_-{\del^{-}_1} \\
D_2 \ar[u]|-{\ep_1} \ar@<1ex>[r]^-{\del^{-}_1} \ar@<-1ex>[r]_-{\del^{+}_1} & D_0 \ar[l]|-{\ep}\ar[u]|-{\ep} \\
}
\]
along with connections $\Ga^{-}, \Ga^{+} \colon D_1 \to D_2$. Two $1$-morphisms $a,b \in D_1$ with same endpoints $\del^{-}_1(a) = \del^{-}_1(b) = x$, $\del^{+}_1(a) = \del^{+}_1(b) = y$ are called \emph{homotopic} if there exists a $2$-morphism $u \in D_2$ satisfying $\del^{-}_2(u) = a$, $\del^{+}_2(u) = b$, $\del^{-}_1(u) = \ep(x)$, $\del^{+}_2(u) = \ep(y)$. We write $a \sim b$ if $a$ and $b$ are homotopic.

We now define the underlying $2$-track groupoid $UD$. The first constituent groupoid $UD_{(1)}$ has object set $D_0$ and morphism set $D_1 / \sim$, with groupoid structure inherited from the groupoid $(D_0,D_1)$. The second constituent groupoid $UD_{(2)}$ has object set
\[
UD_{(2)0} := \left\{ a \in D_1 \mid \del^{+}_1(a) = x_0 \right\}.
\]
A morphism in $UD_{(2)}$ from $a$ to $b$ is an element $u \in D_2$ satisfying $\del^{-}_1(u) = a$, $\del^{-}_2(u) = b$, $\del^{+}_1(u) = \ep(x_0)$, $\del^{+}_2(u) = \ep(x_0)$, as illustrated here:

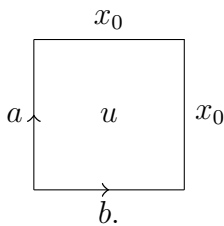
\begin{figure}[H]
\begin{tikzpicture}
\draw [->-=0.5] (0,0) -- (0,2);
\draw [->-=0.5] (0,0) -- (2,0);
\draw (0,2) -- (2,2) -- (2,0);
\node [left] at (0,1) {$a$};
\node [below] at (1,0) {$b$.};
\node [above] at (1,2) {$x_0$};
\node [right] at (2,1) {$x_0$};
\node at (1,1) {$u$};
\end{tikzpicture}
\end{figure}

Composition in $UD_{(2)}$ is defined as follows. Given $1$-morphisms $a,b,c \colon x \to x_0$ in $D_1$ and morphisms $u \colon a \Ra b$ and $v \colon b \Ra c$ in $UD_{(2)}$, their composition $v \sq u \colon a \Ra c$ is defined by
\begin{align*}
v \sq u &= \left( \Ga^{+}(b) +_2 u \right) +_1 \left( v +_2 \odot_{x_0} \right) \\
&= \left( \Ga^{+}(b) +_2 u \right) +_1 v \\
&= \left( \Ga^{+}(b) +_1 v \right) +_2 u
\end{align*}
as illustrated here:

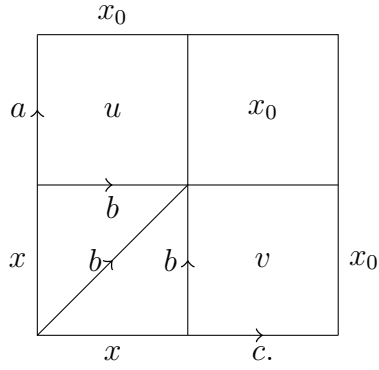
\begin{figure}[H]
\begin{tikzpicture}
\draw [->-=0.5] (0,2) -- (0,4);
\draw [->-=0.5] (0,2) -- (2,2);
\draw [->-=0.5] (2,0) -- (2,2);
\draw [->-=0.5] (2,0) -- (4,0);
\draw [->-=0.5] (0,0) -- (2,2);
\draw (0,4) -- (4,4) -- (4,0);
\draw (0,2) -- (0,0) -- (2,0);
\draw (2,4) -- (2,2) -- (4,2);
\node [left] at (0,3) {$a$};
\node [below] at (1,2) {$b$};
\node [left] at (1,1) {$b$};
\node [left] at (2,1) {$b$};
\node [below] at (3,0) {$c$.};
\node [above] at (1,4) {$x_0$};
\node at (3,3) {$x_0$};
\node [right] at (4,1) {$x_0$};
\node at (1,3) {$u$};
\node at (3,1) {$v$};
\node [left] at (0,1) {$x$};
\node [below] at (1,0) {$x$};
\end{tikzpicture}
\end{figure}

The identity morphisms in $UD_{(2)}$ are given by $\id^{\sq}_a = \Ga^{-}(a)$. The inverse of $u \colon a \Ra b$ is given by
\begin{align*}
u^{\inv} &= \left( (-_1) \Ga^{+}(b) +_2 (-_1) u \right) +_1 \left( \ep_2(a) +_2 \Ga^{-}(a) \right) \\
&= \left( (-_1) \Ga^{+}(b) +_2 (-_1) u \right) +_1 \Ga^{-}(a)
\end{align*}
as illustrated here:

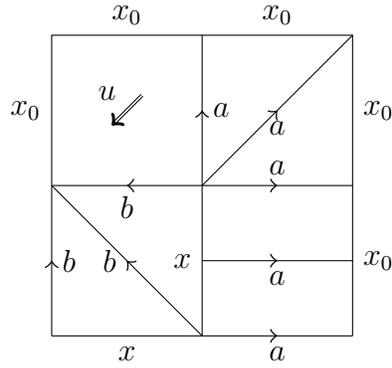
\begin{figure}[H]
\begin{tikzpicture}
\draw [->-=0.5] (2,2) -- (0,2);
\draw [->-=0.5] (2,2) -- (2,4);
\draw [->-=0.5] (2,2) -- (4,2);
\draw [->-=0.5] (0,0) -- (0,2);
\draw [->-=0.5] (2,0) -- (4,0);
\draw [->-=0.5] (2,0) -- (0,2);
\draw [->-=0.5] (2,2) -- (4,4);
\draw [->-=0.5] (2,1) -- (4,1);
\draw [->,double] (1.2,3.2) -- (0.8,2.8);
\draw (0,2) -- (0,4) -- (4,4) -- (4,0);
\draw (0,0) -- (2,0) -- (2,2);
\node [below] at (1,2) {$b$};
\node [right] at (2,3) {$a$};
\node [left] at (1,1) {$b$};
\node [right] at (0,1) {$b$};
\node [above] at (3,2) {$a$};
\node [below] at (3,0) {$a$};
\node [below] at (3,1) {$a$};
\node [below] at (3,3) {$a$};
\node [left] at (0,3) {$x_0$};
\node [above] at (1,4) {$x_0$};
\node [above] at (3,4) {$x_0$};
\node [right] at (4,3) {$x_0$};
\node [right] at (4,1) {$x_0$};
\node [above left] at (1,3) {$u$};
\node [below] at (1,0) {$x$};
\node [left] at (2,1) {$x$};
\end{tikzpicture}
\end{figure}

The structural isomorphisms $\psi_a^{-1} \colon \Aut(\ep(x_0)) \to \Aut(a)$ are defined by
\begin{align*}
\psi_a^{-1}(u) &= \left( \Ga^{-}(a) +_2 \odot_{x_0} \right) +_1 \left( \odot_{x_0} +_2 u \right) \\
&= \Ga^{-}(a) +_1 u \\ 
&= \Ga^{-}(a) +_2 u 
\end{align*}
as illustrated here:

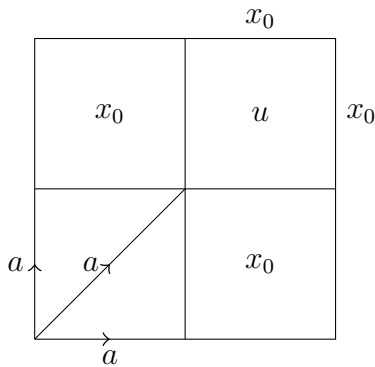
\begin{figure}[H]
\begin{tikzpicture}
\draw [->-=0.5] (0,0) -- (0,2);
\draw [->-=0.5] (0,0) -- (2,2);
\draw [->-=0.5] (0,0) -- (2,0);
\draw (0,2) -- (0,4) -- (4,4) -- (4,0) -- (2,0);
\draw (0,2) -- (4,2);
\draw (2,0) -- (2,4);
\node [left] at (0,1) {$a$};
\node [left] at (1,1) {$a$};
\node [below] at (1,0) {$a$};
\node at (1,3) {$x_0$};
\node at (3,1) {$x_0$};
\node [above] at (3,4) {$x_0$};
\node [right] at (4,3) {$x_0$};
\node at (3,3) {$u$};
\end{tikzpicture}
\end{figure}

The quotient function $q \colon UD_{(2)0} \surj \Star UD_{(1)}$ is induced by the quotient function $D_1 \surj D_1 / \sim$. One readily checks that $UD$ is a $2$-track groupoid, and that this construction $U$ is functorial.

For a pointed space $X$, define a comparison map $\Pi_{(1,2)}(X) \to U \rho_2^{\sq}(X)$ which is an isomorphism on $\Pi_{(1)}(X)$, and which quotients out the thin homotopy relation between left paths in $X$ and the cubically thin homotopy relation between left $2$-tracks. This defines a natural weak equivalence of $2$-track groupoids.
\end{proof}

\end{document}